\documentclass[11pt]{elsarticle}
\usepackage[utf8]{inputenc}
\usepackage{amsthm}
\usepackage{amsmath}
\usepackage{amssymb}
\usepackage{mathtools}
\usepackage{enumerate}
\usepackage{cancel}
\usepackage{enumerate}
\usepackage{mathrsfs}
\usepackage{tikz-cd}
\usepackage{hyperref}
\usepackage{setspace}
\usepackage[top=1in, left=1in, right=1in, bottom=1in, footskip=0.5in]{geometry}
\usepackage{graphicx}
\usepackage{scrextend}
\usepackage{lipsum}
\usepackage{caption}
\usepackage{comment}
\usepackage{fancyhdr}
\usepackage{optidef}
\usepackage{lineno}
\usepackage{booktabs}
\fancypagestyle{plain}{
\fancyhf{}
\fancyfoot[C]{\large\thepage}

}
\usepackage{xcolor}
\usepackage{pgffor}
\usepackage{graphicx}
\usepackage{graphbox}
\usepackage{soul}
\usepackage{color}

\definecolor{greene}{rgb}{0,0.6,0}

\DeclareMathOperator{\Z}{Z}

\DeclareMathOperator{\ft}{ft}

\makeatletter
\def\ps@pprintTitle{%
  \let\@oddhead\@empty
  \let\@evenhead\@empty
  \def\@oddfoot{\reset@font\hfil\thepage\hfil}
  \let\@evenfoot\@oddfoot
}
\makeatother

\newtheorem{theorem}{Theorem}[section]
\newtheorem{lemma}[theorem]{Lemma}
\newtheorem{corollary}[theorem]{Corollary}
\newtheorem{observation}[theorem]{Observation}
\newtheorem{conjecture}[theorem]{Conjecture}
\newtheorem{proposition}[theorem]{Proposition}
\newtheorem{question}[theorem]{Question}

\newtheorem*{claim*}{Claim}
\theoremstyle{definition}
\newtheorem{definition}[theorem]{Definition}
\newtheorem{example}[theorem]{Example}
\newtheorem{remark}[theorem]{Remark}

\numberwithin{figure}{section}   
\numberwithin{table}{section}   
\numberwithin{equation}{section}  

\newcommand{\abs}[1]{\left\vert#1\right\vert}

\newcommand{\rank}{\operatorname{rank}}

\newcommand{\nul}{\operatorname{null}}

\newcommand{\G}{\mathcal{G}}
\newcommand{\F}{\mathcal{F}}

\newcommand{\PP}{\mathfrak{P}}

\newcommand{\R}{\mathbb{R}} 
\newcommand{\Rnn}{\R^{n\times n}} 

\newcommand{\bx}{{\bf x}}
\newcommand{\bv}{{\bf v}}

\newcommand{\bzero}{{\bf 0}}

\newcommand{\mr}{\operatorname{mr}}

\newcommand{\M}{\operatorname{M}}
\newcommand{\n}{\operatorname{N}}

\newcommand{\PC}{\operatorname{P}}
\newcommand{\PAR}{\mathcal{P}}
\newcommand{\LL}{\mathcal{L}}
\newcommand{\T}{\mathcal{T}}

\newcommand{\x}{\times}

\newcommand{\bit}{\begin{itemize}}
\newcommand{\eit}{\end{itemize}}
\newcommand{\ben}{\begin{enumerate}}
\newcommand{\een}{\end{enumerate}}
\newcommand{\beq}{\begin{equation}}
\newcommand{\eeq}{\end{equation}}
\newcommand{\bea}{\begin{eqnarray*}}
\newcommand{\eea}{\end{eqnarray*}}
\newcommand{\bean}{\begin{eqnarray}}
\newcommand{\eean}{\end{eqnarray}}
\newcommand{\bpf}{\begin{proof}}
\newcommand{\epf}{\end{proof}\ms}
\newcommand{\bmt}{\begin{bmatrix}}
\newcommand{\emt}{\end{bmatrix}}
\newcommand{\ms}{\medskip}

\newcommand{\beqa}{\begin{array}}
\newcommand{\eeqa}{\end{array}}
\newcommand{\OL}{\overline}

\newcommand{\lf}{\left\lfloor}
\newcommand{\rf}{\right\rfloor}
\newcommand{\lp}{\left(}
\newcommand{\rp}{\right)}

\newcommand{\cp}{\,\Box \,}

\title{
Forts, (fractional) zero forcing, and Cartesian products of graphs
}
\author[inst1]{Thomas R. Cameron}
\affiliation[inst1]{organization={Department of Mathematics, Penn State Behrend},
            city={Erie},
            postcode={16563}, 
            state={PA},
            country={USA}}

\author[inst2,inst3,inst4]{Leslie Hogben}
\affiliation[inst2]{organization={American Institute of Mathematics},
            addressline={Caltech 8-32; 1200 E California Blvd.},
            city={Pasadena},
            postcode={91125}, 
            state={CA},
            country={USA}}
\affiliation[inst3]{organization={Department of Mathematics, Iowa State University},
            city={Ames},
            postcode={50011}, 
            state={IA},
            country={USA}}
\affiliation[inst4]{organization={Department of Mathematics, Purdue University},
            city={West Lafayette},
            postcode={47907}, 
            state={IN},
            country={USA}}
            
\author[inst5,cor1]{Franklin H. J. Kenter\corref{cor1}}
\ead{kenter@usna.edu}
\affiliation[inst5]{organization={Mathematics Department, United States Naval Academy},
            city={Annapolis},
            postcode={21402}, 
            state={MD},
            country={USA}
}

\author[inst6]{Seyed Ahmad Mojallal}
\affiliation[inst6]{organization={Department of Mathematics and Statistics, University of Regina},
            city={Regina},
            postcode={S4S0A2}, 
            state={SK},
            country={Canada}}

\author[inst7]{Houston Schuerger}
\affiliation[inst7]{organization={Department of Mathematics, The University of Texas Permian Basin},
            city={Odessa},
            postcode={79762}, 
            state={TX},
            country={USA}}

\cortext[cor1]{Corresponding author: kenter@usna.edu}

\begin{document}

\begin{abstract} 
 The fort hypergraph, the  fort number, and the fractional zero forcing number are introduced.  The fort number and the fractional zero forcing number are determined for well-known  graph families and Vizing-like lower bounds are established for these  parameters.    Results on hypergraph transversals and matchings and their fractional versions are applied to the zero forcing number,  fort number, and fractional zero forcing number, establishing a Vizing-like lower bound for the zero forcing number of a Cartesian product of graphs for certain families of graphs. A family of graphs achieving this lower bound is exhibited.

\end{abstract}


\maketitle


{\section{Introduction}

In this work, 
 we establish 
lower bounds for the zero forcing number of the Cartesian product of graphs
under certain conditions. \emph{Zero forcing} is a process on a graph, where vertices are either filled or unfilled.
An initial set of filled vertices can force unfilled vertices to become filled by applying a color change rule. 
While there are many color change rules (see~\cite[Chapter 9]{HogLinShad}), we will use the \emph{(standard)  zero forcing color change rule} which states that a filled vertex $u$ can change an unfilled vertex $w$ to filled if $w$ is the only unfilled neighbor of $u$; this is referred to as \emph{$u$ forcing $w$} and repeated application of this color change rule is referred to as a zero forcing process.
Since the vertex set of a graph is finite, there comes a point in which no more forcings are possible.
If at this point all vertices of the graph $G$ are filled, then we say that the initial set of filled vertices is a \emph{zero forcing set} of $G$.
The \emph{zero forcing number} of $G$, denoted $\Z(G)$, is the minimum cardinality of a zero forcing set of $G$. 

 One of the original applications for zero forcing is to bound the maximum nullity
over all symmetric matrices associated with a graph.  More precisely, given a graph $G$ with $V(G)=\{1,2,\ldots,n\}$ and edge set $E(G)$,  the \emph{set of symmetric matrices associated with $G$} is defined by $\mathcal{S}(G) =$ $\left\{A=[a_{ij}]\in{S}_{n}(\mathbb{R})\colon\forall i\neq j, a_{ij}\neq 0 \Leftrightarrow \{i,j\}\in E(G)\right\}$  where ${S}_{n}(\mathbb{R})$ denotes the set of $n\times n$ real symmetric matrices. The \emph{maximum nullity} of $G$ is defined by $\M(G) = \max\left\{\nul{A}\colon A\in\mathcal{S}(G)\right\},$ where $\nul{A}$ denotes the nullity of $A$. 
The original study \cite{aim} showed that the zero forcing number is an upper bound for the maximum nullity: $M(G) \le \Z(G)$.
{The \emph{Cartesian product} of two graphs $G$ and $G'$, denoted $G \Box G'$, has vertex set $V(G) \times V(G')$, where $A\times B=\{(a,b):a\in A, b\in B\}$, and edge set $\{ (g_1,g'_1)(g_2,g'_2) ~\colon g'_1=g'_2 \mbox{ and } g_1g_2 \in E(G) \} \cup \{ (g_1,g'_1)(g_2,g'_2) ~\colon g_1=g_2 \mbox{ and } g'_1g'_2 \in E(G') \}$.
That original study gave an upper bound for the zero forcing number of the Cartesian product of graphs \cite{aim}: 
$\Z(G \Box G') \le \min\{ \Z(G) |V(G’|, |V(G)| \Z(G') \}$. This upper bound is achieved simply by constructing a zero forcing set of $G \Box G'$ using the zero forcing sets of $G$ or $G'$. However, a {\it lower bound} is more elusive as proving a lower bound requires showing that no smaller subset of vertices can possibly be a zero forcing set.

To approach a lower bound,  one can appeal to similar bounds for the maximum nullity of a graph, as given by Hogben, Lin and Shader. 

\begin{theorem}{\rm \cite[Theorem 9.22,  Corollary 9.23]{HogLinShad}}\label{Thm1}
Let $G$ and $G'$ be graphs each of which has an edge. Then
\[ \Z(G\Box G')\ge M(G\Box G')\ge M(G)M(G')+1.\] 
If $\M(G)=\Z(G)$ and $\M(G')=\Z(G')$, then \[\Z(G\cp G')\ge \Z(G)\Z(G')+1.\] 
 All bounds are sharp.
\end{theorem}

It remains to determine whether the second bound for $\Z(\cdot)$ in Theorem \ref{Thm1} holds for all graphs. This question was originally asked by Hogben, Lin and Shader in \cite{HogLinShad} and we present it here as a formal conjecture.

\begin{conjecture}\label{con:main}
 If $G$ and $G'$ are graphs each containing an edge, then $\Z(G \Box G')\geq \Z(G) \Z(G')+1$.
\end{conjecture}

If true, this bound would be sharp. We provide examples of large families of graphs that obtain equality in Section \ref{s:realize-lowerbd}.

 These types of bounds on graph products, known as Vizing-like bounds, are not uncommon in graph theory. Vizing conjectured that $\gamma(G \Box G') \ge \gamma(G) \gamma(G')$ where $\gamma(\cdot)$ is the domination number \cite{vizing1968some}.   Similar bounds have been found for other graph parameters, e.g., \cite{PPS12}.





 As a tool to study  Conjecture \ref{con:main}, we develop an alternative view of zero forcing using hypergraph transversals. A \emph{hypergraph transversal} (or edge cover) is a set of vertices that intersects every edge. We will let $\tau(H)$ denote the \emph{transversal number} of the hypergraph $H$, which is the  cardinality of the minimum transversal.  Brimkov, Fast and Hicks first gave this alternative characterization of zero forcing in the context of forts \cite{fortthm}. A {\it fort} is a nonempty subset of vertices whereby whenever its complement is filled, no more forces are possible. Specifically, a fort is a nonempty subset $F \subseteq V(G)$ such that for each $v \not \in F$, $\abs{N_G(v) \cap F}\neq 1$, where $N_G(v)$ or $N(v)$ denotes the open neighborhood of $v$.

\begin{theorem}[See \cite{fortthm} and \cite{fort}] \label{thm.fortchar}\label{fortsiffzfs}
A set of vertices $S \subseteq V(G)$ is a zero forcing set if and only if $S$ intersects every fort.
\end{theorem}

With this characterization of zero forcing, we focus on the \emph{hypergraph of forts}, denoted $\F_G$. (We will more formally define $\F_G$ in Section \ref {s:forthypergraph+fractional}). From this  perspective, Theorem \ref{thm.fortchar} can be reinterpreted as follows.

\begin{theorem} \label{thm.fortz} For any graph $G$,
\[ \Z(G) = \tau(\F_G). \]
\end{theorem}

This characterization allows the zero forcing number to be computed as the optimal value of the following binary integer program.  As presented in Model 2 in \cite{fortthm}, all forts are used, but the result is the same when only minimal forts are used, as in the next description.

\begin{mini!}
	{}{\sum_{v\in V(G)}x_{v}}{}{}\label{eq:trans-obj}
	\addConstraint{\sum_{v\in F}x_{v}}{\geq 1,~\quad\text{ for all minimal forts }  F}\label{eq:trans-const1}
    \addConstraint{x_{v}}{\in\{0,1\},~\quad\forall v\in V(G).}\label{eq:trans-const2}
 \end{mini!}

One advantage of this view, not yet taken before, is that it provides a meaningful 
interpretation of a {\it fractional zero forcing number}. This is done by changing the constraint (\ref{eq:trans-const2}) from ${x_{v}} \in\{0,1\}$ to ${x_{v}}{\in [0,1]}$. We will denote the resulting optimal value of this relaxation as $\Z^*(G).$  Note that this fractional variation of zero forcing is different from a previous one introduced in [13] via a three-color forcing game, which was shown to be equal to the skew forcing number $\Z^{-}(G)$.  This distinction is illustrated in Example~\ref{ex:complete} where it is shown that $\Z^{*}(K_{n})=\frac n 2$, since $\Z^{-}(K_{n})=n-2$.

Returning to our main question, Conjecture \ref{con:main}, we leverage known results about Cartesian products of hypergraph transversals. Namely, we 
show in Theorem \ref{thm.zequalzstartrue} that  $\Z(G \Box G') \ge \Z(G) \Z(G') +1$ whenever $\Z(G) = \Z^*(G)$. As a result, we can provide an affirmative answer to Conjecture \ref{con:main} for many families of graphs (by showing that $\Z(G) = \Z^*(G)$ in those cases).
  It is interesting to compare Theorem \ref{thm.zequalzstartrue} to Theorem \ref{Thm1}:  While more graphs $G'$ are known for which $\Z(G') = \M(G')$ than for which $\Z(G') = \Z^*(G')$, once such a graph $G'$ is found it has wider application to Cartesian products: By Theorem \ref{thm.zequalzstartrue}, the bound $\Z(G \Box G') \ge \Z(G)\Z(G') + 1$ applies without restriction on $G$, whereas to obtain $\Z(G \Box G') \ge \Z(G)\Z(G') + 1$ from Theorem \ref{Thm1} requires $\Z(G)=\M(G)$ in addition to $\Z(G')=\M(G')$.  
For $G=C_5\circ K_1$ (called the pentasun),  it is well known that that $Z(G)=3>2=M(G)$.  So if $G'$ is any graph with $\Z(G') = \Z^*(G')$ (examples of such graphs are listed  in Table \ref{table2}), then Theorem \ref{thm.zequalzstartrue} confirms the bound in Conjecture \ref{con:main}, while Corollary \ref{Thm1} does not.

While applying hypergraph results, we also prove the corresponding Vizing-like bounds for all graphs $G$ and $G'$ for the parameters $\Z^*$ and  the additional new parameter $\ft$,    which is the maximum cardinality of a collection of disjoint forts in a given graph and is called the \emph{fort number}. Specifically,   in Proposition \ref{prop:cart_tens_prod_fractional} and Corollary \ref{boxfort} we show that for each pair of graphs $G$ and $G'$,
\[\ft(G \Box G') \ge \ft(G)\ft(G') \text{ and }\Z^*(G\Box G')\ge\Z^*(G)\Z^*(G').\]

 Additional contributions of this study include the following:

\begin{itemize}

\item A determination of the exact values of these parameters, including $\Z^*$, for many families of graphs including complete graphs, complete bipartite graphs,  cycles, hypercubes, select coronas and more (see Section \ref{s:fam+ext}; these results are summarized in Table \ref{t:ftzstarzvalues}). 

\item A full characterization  when $\Z^*(G) = \Z(G)$ for trees (see Theorem \ref{thm:LL}).

\item A result that equality within Conjecture \ref{con:main} is achieved whenever $G$ is a star-clique path (see Section \ref{s:realize-lowerbd}).
  
\item   Numerous interesting open questions covering many different aspects of this study (summarized in Section \ref{subsec:open} and Table \ref{tab:openquestions}).
\end{itemize}

\section{Forts,  hypergraph transversals, and fractional zero forcing} \label{s:forthypergraph+fractional}
In this section we more carefully define the fort hypergraph and use it to define a (new) fractional zero forcing number and fractional fort number, which are equal.
These ideas are then used in Section \ref{sec:cart_prod_bounds} to leverage known results on hypergraph matchings and hypergraph transversals. 
    
Let $H=(V(H),E(H))$ be a {hypergraph}, where $V(H)$ is a finite nonempty set (called the set of vertices) and the set of edges $E(H)$ is a set of nonempty subsets of $V(H)$. As in~\cite{BergeHypergraphs}, we require that $V(H)=\cup_{e\in E(H)} e$.  The \emph{degree} of a vertex in a hypergraph is the number of edges that it is an element of. A set of vertices that intersects every edge of $H$ is a \emph{transversal} of $H$.
The \emph{transversal number} of  $H$, denoted by $\tau(H)$,
is the minimum cardinality of a transversal of $H$. 
A hypergraph is \emph{simple} if no edge is a proper subset of another edge,  \emph{uniform} if every edge has the same number of vertices, and \emph{regular} if every vertex is in the same number of edges.  

Throughout, we will consider the  hypergraph of minimal forts.
One might think it would be natural to consider the hypergraph of all forts rather than the hypergraph of minimal forts. 
Such a fort hypergraph is also a reasonable object of study. 
However, a transversal of all minimal forts is also a transversal of all forts (because every fort contains a minimal fort) and it is convenient to consider only minimal forts.  
Furthermore, by restricting consideration to minimal forts, a simple hypergraph is obtained.

 \begin{definition}\label{def-forthyp}
    Let $G$ be a graph. The \emph{hypergraph of minimal forts} or \emph{fort hypergraph} of $G$, denoted $\F_G$, is the hypergraph with  $E(\F_G) = \{ F : F \text{ is a minimal fort of } G\}$ and $V(\F_G) = \cup_{F\in E(\F_G)} F$.
\end{definition}
Note that a vertex of $G$ that is not in any minimal fort of $G$ will not be a vertex of $\F_G$. However, if $v$ is an isolated vertex of $G$, then $\{v\}$ is a minimal fort of $G$, so $v\in V(\F_G)$. A fort hypergraph is simple but need not be uniform nor regular.

For convenience, we may abuse notation and use $\F_G$ to refer to either the set of minimal forts of $G$ or the hypergraph of minimal forts of $G$. In all cases, the context will be clear.
As mentioned in Theorem \ref{thm.fortchar}, the strong relationship between zero forcing and hypergraph transversals is that $\Z(G) = \tau(\F_G)$.


    Given a hypergraph $H$, a set of disjoint edges is a  \emph{matching} of $H$. A \emph{maximum matching} is a matching of maximum cardinality.
   The \emph{matching number} of a hypergraph $H$, denoted here by $\mu(H)$
is the number of edges in a maximum matching.\footnote{The matching number of $H$ is often denoted by $\nu(H)$ in the literature.} 

\begin{observation} \label{obs.fteqmatch} For any graph $G$,
\[\ft(G) = \mu(\F_{G}).\]
\end{observation}

For any hypergraph $H$, $\mu(H) \le \tau(H)$ (as one will need at least one vertex from each edge in the matching to cover those edges). In the context of zero forcing,
$\mu(\F_G) \le \tau(\F_G)$ 
translates to 
$\ft(G)\le\Z(G)$.

In analogy with equations \eqref{eq:trans-obj} -- \eqref{eq:trans-const2}, the matching problem for a hypergraph $H$ can be formulated as follows:
 \begin{maxi!}
    {}{\sum_{e\in E(H)}x_{e}}{}{}\label{eq:match-obj}
    \addConstraint{\sum_{e\in E(H)\colon v\in e}x_{e}}{\leq 1,~\quad\forall v\in V(H)}\label{eq:match-const1}
    \addConstraint{x_{e}}{\in\{0,1\},~\quad\forall e\in E(H).}\label{eq:match-constr2}
 \end{maxi!}

 The fractional transversal number and fractional matching number are discussed in~\cite[Chapter 3]{BergeHypergraphs} and~\cite[Chapter 1]{Scheinerman2008}. The fractional transversal problem for a hypergraph $H$ can be formulated as a linear program obtained by relaxing constraint~\eqref{eq:trans-const2} to $x_{v}\in[0,1]$, for all $v\in V(H)$.  
The fractional matching problem for a hypergraph $H$ is the dual linear program obtained by relaxing constraint~\eqref{eq:match-constr2} to $x_{e}\in[0,1]$, for all $e\in E(H)$.   An assignment of values  $x_v\in [0,1]$ to vertices is called a \emph{weighting of vertices} and an assignment of values  $x_F\in [0,1]$ to forts (edges of the fort hypergraph) is called a \emph{weighting of forts}.

From a graph-theoretical sense, the transversal problem seeks a set of vertices of minimum cardinality that intersect every edge.
Whereas, the fractional-transversal problem seeks to assign a minimum total weight to the vertices such that each edge has weight at least $1$.
Similarly, the matching problem seeks a maximum collection of disjoint edges.
Whereas, the fractional-matching problem seeks to assign a maximum total weight to the edges such that the edges containing a single vertex have a combined weight of at most $1$.

Given a hypergraph $H$, we denote its fractional transversal number by $\tau^{*}(H)$ and its fractional matching number by $\mu^{*}(H)$.
From here, we can define the fractional zero forcing number and fractional fort number of a graph.

\begin{definition}\label{d:fracZ}
For a graph $G$, the \emph{fractional zero forcing number} is defined by
\[
\Z^{*}(G) = \tau^{*}(\F_G)
\]
or
\[ 
\Z^{*}(G) = \min
\left\{\sum_{v\in V(G)}x_{v}\colon\sum_{v\in F} x_{v}\geq 1~\forall F\in\F_G \text{ and } x_{v}\geq 0~ \forall v\in V(G)\right\}
\]
where $\F_G$ is the hypergraph or set of minimal forts,   respectively.
\end{definition}

\begin{definition}\label{d:fracft}
For a graph $G$, the \emph{fractional fort number} is defined by
\[
\ft^{*}(G) = \mu^{*}(\F_{G})
\]
or
\[
\ft^{*}(G) = \max\left\{\sum_{F\in\F_{G}}x_{F}\colon\sum_{F\in\F_{G}\colon v\in F}x_{F}\leq 1,~\forall v\in V(G)\text{ and } x_{F}\geq 0~\forall F\in\F_{G}\right\}
\]
where $\F_G$ is the hypergraph or set of minimal forts, respectively.
\end{definition}

\begin{remark}
Since the relaxations of the integer programs in~\eqref{eq:trans-obj}--\eqref{eq:trans-const2} and~\eqref{eq:match-obj}--\eqref{eq:match-constr2} are dual linear programs, the duality theorem of linear programming implies that $\tau^{*}(H) = \mu^{*}(H)$~\cite[Corollary 7.1g]{schrijver}. 
Thus for all graphs $G$, we have the following:
\begin{equation}\label{eq:ftz-ineq}
\ft(G)\leq\ft^{*}(G)=\Z^{*}(G)\leq\Z(G).
\end{equation}

\end{remark}

 It  is sometimes convenient to  denote a weighting of vertices $v\to x_v$ by a \emph{weight function} $\omega:V(G)\to [0,1]$ with $\omega(v)=x_v$.  We say a weight function $\omega$ is \emph{valid} for $G$ if it satisfies the constraint 
$\sum_{v\in F} x_{v}\geq 1$ for all $F\in\F_G$.  An \emph{optimal} weight function is a valid weight function $\omega$ such that $\sum_{v\in V(G)} \omega(v)=\Z^*(G)$. 


{Most variations of zero forcing are introduced using an alternative color-change rule. However, a natural color-change rule for $\Z^*(G)$ appears elusive. Hence, we ask the following open question.

\begin{question} \label{quest:ccr_zstar}
Is there a color-change rule that can be used to compute $\Z^*$?
\end{question}
}


\section{Fort number and fractional zero forcing number of families of graphs and extreme values}\label{s:fam+ext}
We begin  this section by determining the  fort number and fractional zero forcing number of several families of graphs, including (but not limited to) paths, cycles, complete graphs, and complete bipartite graphs; these results are summarized in Table \ref{t:ftzstarzvalues}.   We also introduce some tools for computing fractional zero forcing number. Finally, we examine graphs with the smallest and largest possible fractional zero forcing numbers (among connected graphs of fixed order).
\begin{example}\label{ex:path}Let $P_{n}$ denote a path graph of order $n$. Since $\Z(P_{n})=1$  (and $\ft(G)\ge 1$ for every graph $G$), it follows that\[\ft(P_{n}) = \Z^{*}(P_{n}) =\Z(P_{n}) = 1.\]  For $n\ge 4$, $\F_{P_n}$  is never regular (since each of the two leaves is in every fort but other vertices are not). 
\end{example}

The next lemma is useful for graphs that have forts with cardinality two, including complete graphs.
 In a graph $G$, we say that vertices $u$ and $w$ are \emph{twins} if $v\in N_G(u)$ if and only if $v\in N_G(w)$ for $v\ne u,w$.

\begin{lemma}\label{l:F2}
Let $G$ be a graph, then the following hold:
   \begin{itemize} 
   \item $F_0=\{u,w\}$ is a fort of $G$ if and only if $u$ and $w$ are twins.   
   \item If  $u$ and $w$ 
   are twins and $F$ is a fort of $G$ such that  $u\in F$ and $w\not\in F$, then $F\setminus\{u\}\cup\{w\}$ is also a fort of $G$. 
    \item Whenever $u$ and $w$ are twins and there is a fort $F$ with $u\in F$ but $w\not\in F$, every optimal weight function has $\omega(u)=\omega(w)= \frac 1 2$.
     \end{itemize}
\end{lemma}

\begin{proof} 
Observe that $F_0=\{u,w\}$ is a fort of $G$ if and only if $ v\in N(u) \Leftrightarrow v\in N(w)$ for every $v\ne u,w$ if and only if $u$ and $w$ are twins.  

Now assume $u$ and $w$ are twins, $F$ is a fort of $G$, $u\in F$, and $w\not\in F$. Let $F'=F\setminus\{u\}\cup\{w\}$ and let $v\in V(G)\setminus{F'}$.
If $v\neq u$, then $\abs{N(v)\cap F'} = \abs{N(v)\cap F} \neq 1$. If $v=u$, then $\abs{N(v)\cap F'} = \abs{N(w)\cap F} \neq 1$. Therefore, $F'$ is a fort of $G$. 
Since $F_0$ is a fort, $\omega(u)+\omega(w)\geq 1$.  Since $F$ and $F'$ are both forts of $G$ it follows that $\omega(u)=\omega(w)=1/2$ for any optimal weight function $\omega$. \end{proof}

    If every minimal fort has exactly 2 vertices, then  $\F_G$ is a 2-uniform hypergraph and thus is a graph, as is the case in the next example.  

\begin{example}\label{ex:complete}
Let $K_{n}$ denote the complete graph of order $n$, where $n\geq 2$.
Recall that $\Z(K_{n}) = n-1$.
Also, observe that the minimal forts of $K_{n}$ are made up of any pair of two vertices. 
Therefore, $\ft(K_{n})=\lf\frac{n}{2}\rf$.
Furthermore, Lemma~\ref{l:F2} implies that an optimal weight function has $\omega(u)=1/2$ for all vertices $u\in V(K_{n})$.
Thus, $\Z^{*}(K_{n}) =  \frac n 2$. Observe that the fort hypergraph of the complete graph is  a graph and $\F_{K_n}\cong K_n$.
\end{example}

\begin{example}\label{ex:Kpq}
Consider the complete bipartite graph $K_{p,q}$ where we have a partition of the vertices $V_{1},V_{2}$ with $\abs{V_{1}}=p$, $\abs{V_{2}}=q$, and $uv\in E$ if and only if $u\in V_{1}$ and $v\in V_{2}$.  

Assume first that  $p,q\geq 2$.
Every pair of vertices in $V_{1}$ (or pair in $V_{2}$) constitutes a fort of $K_{p,q}$.
  No set of one vertex from each partite set is a fort.  Therefore, $\ft(K_{p,q}) = \lfloor\frac{p}{2}\rfloor + \lfloor\frac{q}{2}\rfloor$.
Furthermore, since any  set of $2$ vertices with both in $V_1$ or both in $V_2$ is a fort of $K_{p,q}$, Lemma~\ref{l:F2} implies that the optimal weight function has $\omega(u)=1/2$ for all vertices $u\in V(K_{p,q})$.
Thus $\Z^{*}(K_{p,q}) = \frac{p+q}{2}$.

For $p=1$ and $q\ge 2$, the formula for fort number  remains valid: $\ft(K_{1,q})=\lf \frac{q}2\rf$. 
However, $\Z^*(K_{1,q})= \frac{q}2$, because the only minimal forts are pairs of leaves.  
In both cases, the fort hypergraph is 2-uniform, and hence is a graph  (that is disconnected when $p,q\ge 2$). When $p=q$ or $p=1$, then the fort hypergraph is regular.
\end{example}


  \begin{observation}\label{o:disjoint}
    If every pair of distinct minimal forts of $G$ is disjoint, then $\ft(G)=\Z^*(G)=\Z(G)$.
\end{observation}

\begin{example}\label{ex:empty}
Let $\overline{K_n}$ denote the empty graph of order $n$.  Each set consisting of a single vertex is a minimal fort and $\F_{\overline{K_{n}}}$ is 1-uniform and 1-regular. 
From 
Observation \ref{o:disjoint},
$\ft(\overline{K_{n}}) = \Z^{*}(\overline{K_{n}}) =  \Z(\overline{K_{n}})=n$. 
\end{example}

The next remarks provide upper and lower bounds on the fractional zero forcing number that are used in determining its value for various families.

\begin{remark}\label{r:nover2} 
    Suppose that $G$ is a connected graph of order $n\ge 2$. Then each fort of $G$ must contain at least two vertices. Hence, we can weight each vertex $\frac 1 2$ to cover each fort with weight at least $1$, and it follows that $\Z^{*}(G)\leq \frac n 2$.
Furthermore, by Example~\eqref{ex:complete}, this bound is sharp.

This bound can be improved when all forts are larger or some vertices are not in any fort: Let  $\varphi$ be the minimum cardinality of a fort and let $|V(\F_{G})|=n'$.   We see that  $\Z^*(G)\le \frac{n'}{\varphi}$ by considering the valid weight function $\omega(v)=\frac{1}{\varphi}$ for each vertex $v$ in  $V(\F_{G})$ and $\omega(v)=0$ for $v\not\in V(\F_{G})$. 
\end{remark}

We can obtain a lower bound on the fractional fort number even when only some minimal forts of a graph are known.  
\begin{remark}\label{p:symm-fort-ratio}
Let $S$ 
be a set of $s$  minimal forts of $G$, and let $d$ denote the largest number of forts in $S$ that contain any one vertex of $G$ (so $d$ is the maximum degree of a vertex in the subhypergraph with edges $S$).  We see that $\frac s d\le \ft^*(G)=\Z^*(G)$ by weighting each fort in $S$ as $\frac 1 d$ and every other minimal fort weighted zero.
In particular, if $m=m(\F_G)$ is the number of edges in the fort hypergraph $\F_G$ and $\Delta=\Delta(\F_G)$ is the maximum vertex degree in $\F_G$, then $\frac{m}{\Delta}\le \ft^*(G)$.   
\end{remark}

\begin{remark}\label{Z*bds}
By the two previous remarks  (and with the notation used there), 
         \[\frac{m}{\Delta} \le \ft^*(G)=\Z^*(G) \le \frac{n'}{\varphi}.\]      If  $\frac{m}{\Delta}=\frac{n'}{\varphi}$ for a graph $G$, then $\ft^*(G)=\Z^*(G)=\frac{n'}{\varphi}$. In particular, if the fort hypergraph of $G$ is both $\varphi$-uniform and $\Delta$-regular, then there are $n'\Delta = \varphi m$ edge-to-vertex incidences and thus $\ft^*(G)=\Z^*(G)=\frac{n'}{\varphi}=\frac m\Delta$.
\end{remark}

\begin{example}\label{ex:cycle}
Let $C_{n}$ denote a cycle of order $n$, where $n\geq 3$. Recall that $\Z(C_{n})=2$. Number the vertices of $C_n$ as $0,1,\dots,n-1$ in cycle order and perform arithmetic modulo $n$.

First consider the case $n=2k$.  Notice that the disjoint sets $F_e=\{0,2,\dots, 2k-2\}$ and $F_o=\{1,3,\dots, 2k-1\}$ are each forts. Hence, 
$
\ft(C_{2k})=\Z^{*}(C_{2k})=\Z(C_{2k})=2.
$

 Now let $n=2k+1$.  Each fort contains at least $k+1$ vertices, so
$\Z^{*}(C_{2k+1})\leq\frac{2k+1}{k+1}$ by Remark \ref{r:nover2}.  The sets  of the form $F_\ell=\{\ell,\ell+2,\dots, \ell+2k=\ell-1\}$ are all  minimal forts.  There are $2k+1$ such forts and each vertex appears in $k+1$ forts.
Thus,   $\frac{2k+1}{k+1}\le \Z^*(C_{2k+1})$ by Remark \ref{p:symm-fort-ratio}. 
So $\Z^{*}(C_{2k+1})=\frac{2k+1}{k+1}<2=\Z(C_{2k+1})$.  Note also that $\ft(C_{2k+1})=1$  because every fort contains at least $k+1$ vertices.
\end{example}

The next  two examples make use of coronas of graphs. Let $G$ and $G'$ be graphs.
The \emph{corona} of $G$ with $G'$, denoted $G\circ G'$, is obtained by taking one copy of $G$ and $\abs{V(G)}$ copies of $G'$ and joining the $i$th vertex of $G$ to every vertex in the $i$th copy of $G'$.

\begin{example}\label{ex:double_foliation}
Let $G'$ be a graph of order $r$ and let $G=G'\circ 2K_1$.   
Every  set of two leaves adjacent to the same vertex of $G'$ is a minimal fort and $\Z(G)\le r$ since choosing one leaf from each of these minimal forts is a zero forcing set.
Every minimal fort is a set of two leaves adjacent to the same vertex of $G'$ because if a fort $F$ contains a vertex $u\in V(G')$, then $F$ contains the two leaf neighbors of $u$, and so is not minimal.
Thus the minimal forts are disjoint and $
\ft(G)=\Z(G)$.  
\end{example}


\begin{example}\label{ex:triangles} 
Let $G'$ be a graph of order $r\ge 2$, $G=G'\circ K_2$, and $W=V(G)\setminus V(G')$.
Then $r \le \ft(G)$ since each pair of adjacent vertices in $W$ is a fort. 
Since every fort of $G$ contains at least two vertices of $W$, weighting every vertex of $W$ as $\frac 1 2$ and the vertices of $G'$ weighted zero gives $\Z^*(G)\le r$. 
Thus $\ft(G)=\Z^*(G)=r$. 
\end{example}

  The next result is stated  in \cite{ZGcircKs} for connected graphs of any order, so does not exclude $G''= K_1$, but it is not correct when $G''=K_1$ (e.g., $\Z(P_{4r}\circ K_1)=2r$ is a counterexample).  We present a proof that covers the graphs in Example \ref{ex:triangles}. 
The notation $v\to w$ is widely used to indicate $v$ forces $w$ and we use that notation here.  In addition, given $U \subseteq V(G)$ we use the notation $G[U]$ to refer to the subgraph with vertex set $V(G[U])=U$ and edge set $E(G[U])=\{uv \in E(G): u,v \in U\}$ and call this the subgraph of $G$ induced by $U$, or more simply an induced subgraph.

\begin{proposition}\label{p:Zcorona} 
   Let $G'$ and $G''$ be graphs of orders $q$ and $r\ge 2$, respectively, such that $G''$ has no isolated vertices.  Then $\Z(G'\circ G'')=q\Z(G'')+\Z(G')$. 
\end{proposition}
\bpf Let $G=G'\circ G''$, let $V(G')=\{u_1,\dots,u_{q}\}$, let $G''_i$ denote the copy of $G''$ joined to $u_i$,  let $V(G''_i)=\{x^i_1,\dots,x^i_r\}$, and let $\hat G_i=G[V(G''_i)\cup\{u_i\}]$.  
First, choose zero forcing sets $S'$ for $G'$ and $S_i$ for $G''_i$.  Then $S'\cup\bigcup_{i=1}^r S_i$ is a zero forcing set of $G$, alternating forces in $G''_i$ that have $u_i$ filled with forces in $G'$ by $u_i$ such that all vertices of $G''_i$ are filled.  Thus $\Z(G)\le q\Z(G'')+\Z(G')$.

Suppose $S$ is a minimum zero forcing set of $G$. 
Define $S'=S\cap V(G’)$ and $\hat S_i=S\cap V(\hat G_i)$.    By \cite[Proposition 9.16]{HogLinShad}, $\Z(\hat G_i)=\Z(G''_i)+1$.  Since at most one vertex of $\hat G_i$ can be forced by a vertex outside $\hat G_i$, $|\hat S_i|\ge \Z(G'')$.  Since $S$ is a minimum zero forcing set of $G$, $|\hat S_i|\le \Z(G'')+1$.  Let  $I=\{i: |\hat S_i|= \Z(G'')+1\}$ and $B=\{u_i: i \in I\}$.  We show that $B$ is a zero forcing set of  $G'$, which implies that $\Z(G)= |S|\ge q\Z(G'')+\Z(G')$.

There are four possible types of forces:  1) $u_i\to u_j$, which requires all vertices of $G''_i$ to be filled; 2)  $x^i_k\to x^i_j$, which requires  $u_i$ to be filled; 3)  $u_i\to x^i_j$, which requires every vertex of $G''_i$ except $x^i_j$ to be filled; 4) $ x^i_j\to u_i$, which requires every neighbor of $x^i_j$ in $G''_i$ to be filled. 
Since $G''$ has no isolated vertices, a force $u_i\to x^i_j$ can be replaced by $x^i_k\to x^i_j$ where $x^i_k\in N_{G''_i}(x^i_j)$ without affecting any other forces. 
If a force $ x^i_j\to u_i$ takes place, it is the first force involving a vertex in $G''_i$.   Since $G''$ has no isolated vertices, we may replace $S$ by $S\setminus \{x^i_k\}\cup\{u_i\}$ where  $x^i_k\in N_{G''_i}(x^i_j)$, and replace $ x^i_j\to u_i$ by $x^i_j\to  x^i_k$ without affecting any other forces.  
Thus we may assume every force is one of the first two types. Given that only the first two types of forces are performed, $B$ must be a zero forcing set of  $G'$.  
\epf

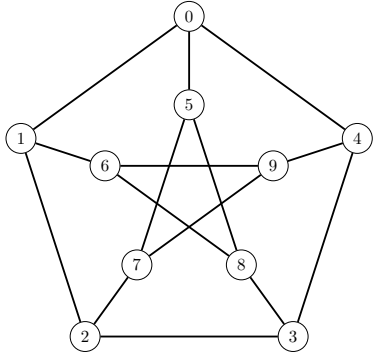
\begin{figure}[ht]
\centering
\begin{tikzpicture}[scale=0.5]
    \node[circle,draw=black,fill=white] (0) at (0,4) {$0$};
    \node[circle,draw=black,fill=white] (1) at (-3.80,1.24) {$1$};
    \node[circle,draw=black,fill=white] (2) at (-2.35,-3.24) {$2$};
    \node[circle,draw=black,fill=white] (3) at (2.35,-3.24) {$3$};
    \node[circle,draw=black,fill=white] (4) at (3.80,1.24) {$4$};
    
    \node[circle,draw=black,fill=white] (5) at (0,2) {$5$};
    \node[circle,draw=black,fill=white] (6) at (-1.90,0.62) {$6$};
    \node[circle,draw=black,fill=white] (7) at (-1.18,-1.62) {$7$};
    \node[circle,draw=black,fill=white] (8) at (1.18,-1.62) {$8$};
    \node[circle,draw=black,fill=white] (9) at (1.90,0.62) {$9$};

    \foreach \x/\y in {0/1,1/2,2/3,3/4,4/0}
        \draw[black,=>latex',-,very thick] (\x) -- (\y);
    \foreach \x/\y in {0/5,1/6,2/7,3/8,4/9}
        \draw[black,=>latex',-,very thick] (\x) -- (\y);
    \foreach \x/\y in {5/7,7/9,9/6,6/8,8/5}
        \draw[black,=>latex',-,very thick] (\x) -- (\y);
\end{tikzpicture}
\caption{Petersen Graph}
\label{petersen}
\end{figure}

\begin{example} \label{ex:petersen}
 Let $P$ be the Petersen graph given in Fig. \ref{petersen}.  The  minimal forts of $P$ are
 $\{0, 1, 3, 8\}, \{0, 1, 9, 7\}, \{0, 2, 3, 5\}, \{0, 2, 4, 7\}, \{0, 8, 2, 9\}, \{0, 3, 6, 7\}, \{0, 8, 4, 6\}, \{0, 9, 5, 6\}, \{1, 2, 4, 9\},$ \\$
 \{8, 1, 2, 5\}, \{1, 3, 4, 6\}, \{1, 3, 5, 9\}, \{8, 1, 4, 7\}, \{1, 5, 6, 7\}, \{9, 2, 3, 6\}, \{2, 4, 5, 6\}, \{8, 2,
6, 7\}, \{3, 4, 5, 7\},$\\$
\{8, 9, 3, 7\}, \{8, 9, 4, 5\}$  \cite{sage}.

Since $\F_P$ is 4-uniform and 8-regular, we have $\Z^*(P)=\ft^*(P)=\frac{20}{8}=2.5=\frac {10}4$ and since we can find two disjoint forts in the list, we also have $\ft(P)=2$.  The zero forcing number of the Petersen graph $P$ is $\Z(P)=5$  \cite{aim}.
\end{example}

Let $Q_d$ $= \underbrace{K_2 \Box K_2 \Box \cdots \Box K_2}_{d \text{ times}}$ denote the \emph{hypercube} of dimension $d$.
 The \emph{failed zero forcing number} of a graph $G$ is the largest cardinality of a set of vertices of $G$ that is not a zero forcing set.  Note that the failed zero forcing number of $G$ is equal to the order minus the minimum cardinality of a fort of $G$, because  $V(G)\setminus F$ is a failed zero forcing set for every fort $F$, and $V(G)\setminus W$ is a fort for every maximal failed zero forcing set $W$.

\begin{proposition}\label{ex:hypercube}
For $d\ge 2$, $\Z^*(Q_d)=\frac {2^d}d$.
\end{proposition}
\bpf 
First we show that the neighborhood $N(v)$ is a fort for any vertex $v$. Observe that the argument is the same for every vertex. We use the $d$-tuple representation with $v=00\dots0$, so $u\in N(v)$ if and only if there is exactly one $1$ in $u$. Then $w\in V(Q_d)\setminus{N(v)}$ has a neighbor in $N(v)$ if and only if $w$ has exactly two 1s or is $v$. In either case, $w$ has at least two neighbors in $N(v)$. 
Then  $\frac{2^d}d\le \Z^*(Q_d)$ by 
Remark \ref{p:symm-fort-ratio} using the $2^d$ forts of the form $N(x)$ for $x\in V(Q_d)$ because each $y\in V(Q_d)$ is in exactly $d$ such forts since $y$ has $d$ neighbors.

By~\cite[Theorem 4.2]{Afzali2024}, the failed zero forcing number of $Q_{d}$ is $2^{d}-d$. 
Therefore, the minimum forts of $Q_{d}$ have cardinality $d$.
Thus $\Z^*(Q_d)\le \frac{2^d} d$ by Remark  \ref{Z*bds}.  
\epf

 The next result shows  that the difference between $\ft(G)$ and $\Z^*(G)$ can be (at least) 
 $\frac n 6$  asymptotically, just as  $K_n$ shows the difference between $\Z(G)$ and  $\Z^*(G)$  can be (at least) $\frac n 2$ asymptotically (where $n$ is the order of $G$).  The join of disjoint graphs $G$ and $G'$ is denoted by $G\vee G'$.

\begin{proposition}\label{ex:sK3veeK1}
   For $s\ge 2$, let $G=sK_3 \vee K_1$ and let $n=3s+1$.  Then $\ft(G)=\frac {n-1}3+1$, $\Z^*(G)=\frac {n-1}2$, and $\Z(G)=\frac{2(n-1)}3+1$.
\end{proposition}

\bpf Let $c$ be the vertex of the $K_1$, so $\deg_G c=3s$. Label the vertices of the $j$th copy of $K_3$ by $x_j,y_j,z_j$.  Then each of $\{x_j,y_j\}$, $\{x_j,z_j\}$, and $\{y_j,z_j\}$ is a fort for $j=1,\dots,s$; we call such a fort a \emph{standard   $2$-fort}.  
Any set of the form $\{c,w_1,\dots,w_s\}$ where $w_j\in\{x_j,y_j,z_j\}$ is a fort; we call such a fort a \emph{standard $(s+1)$-fort}. 
By choosing forts $\{x_j,y_j\}, j=1,\dots,s$ and $\{c,z_1,\dots,z_s\}$, we see that $\ft(G)\ge s+1$.  

Now consider a minimal fort $F$ that is not a standard 2-fort.  Since $F$ must have at least two vertices, it has some vertex other than $c$; without loss of generality, $z_1\in F$.  Note that $F$ cannot contain two elements of $\{x_j,y_j,z_j\}$ for any $j=1,\dots,s$ or it would contain a standard 2-fort.  
Thus $x_1,y_1\not\in F$.  Since $z_1\in N_G(x_1)$, $x_1$ must have another neighbor in $F$, i.e., $c\in F$.  Now for $j=2,\dots, s$, $c\in N(x_j)$, so at least one of $x_j,y_j,z_j$ must be in $F$.  Since $F$ is minimal, $F$ is a standard $(s+1)$-fort.  Thus $\ft(G)=s+1  = \frac{n-1}3+1$. 

 For each $j=1, \dots, s$ and each vertex $w_j \in \{x_j,y_j,z_j\}$, $w_j$ is contained in a standard 2-fort, so by Lemma \ref{l:F2}, $\omega(w_j)=\frac{1}{2}$ for any optimal weight function $\omega$. Starting with $\omega(w_j)=\frac{1}{2}$ for each vertex $w_j \in \{x_j,y_j,z_j\}$ and assigning  $\omega(c)=0$ results in a valid weight function, which is optimal.  Thus $\Z^*(G)=\frac{n-1}2$.

A set  that contains $c$ and two vertices from each copy of $K_3$ is a zero forcing set so $\Z(G)\le 2s+1$.  Any zero forcing set $B$ must contain an element of each fort.  Thus $2s$ elements 
are needed by the standard 2-forts.  With exactly these $2s$ vertices in $B$, there is a standard $(s+1)$-fort that does not contain an element of $B$.  Thus, $\Z(G)\ge 2s+1$.
\epf

 \begin{remark}\label{r:conn}
Observe that fort number and fractional zero forcing number sum over the connected components. Furthermore, isolated vertices and connected components with order at least two behave very differently, because $\ft(K_1)=\Z^*(K_1)=1$, whereas $\ft(G)\le \Z^{*}(G)\leq \frac n 2$ for a connected graph of order $n\ge 2$.
\end{remark}

As is done in the study of zero forcing number,  we can focus on connected graphs of order at least two and then infer results for all graphs by Remark \ref{r:conn}.


Finally, we characterize graphs having the lowest possible value of fractional zero forcing number and highest possible fort number  and fractional zero forcing number.  Note that $\Z^*(G)\ge 1$ for every graph $G$ and this bound is achieved by $P_n$ (cf.  Example \ref{ex:path}). The next result shows that paths are the only graph achieving this lowest value.
\begin{proposition}\label{p:fractZ=1} For a graph $G$, 
$\Z^{*}(G)=1$ if and only if $G$ is a path graph. 
\end{proposition}
\begin{proof}
Suppose that $\Z^{*}(G)=1$. In what follows, we show that there exists a vertex $v\in V(G)$ such that $v\in F$ for all $F\in\F_{G}$.  Then $\{v\}$ intersects every fort and  so is a zero forcing set.  Thus $\Z(G)=1$, which implies $G$ is a path graph  \cite[Theorem 9.12(1)]{HogLinShad}.
To this end, suppose that no such $v\in V(G)$ exists.
Let $\omega$ be an optimal weight function and let $F\in\F_{G}$.
Then, $\sum_{v\in F}\omega(v)\geq 1$ and there exists a vertex $u\in F$ with $\omega(u)>0$.  Since $u$ is not in every fort, there exists a fort $\hat{F}\in\F_{G}$ such that $u\notin\hat{F}$.
Then $\sum_{v\in V(G)}\omega(v)\ge \sum_{v\in\hat{F}}\omega(v)+\omega(u)\geq 1+\omega(u)>1$,  which contradicts $\Z^{*}(G)=1$.
\end{proof}

Note that $\ft(G)\ge 1$ for every graph $G$. This bound is achieved by $P_n, C_{2k+1}$, and numerous other graphs. 

By Remark \ref{r:nover2}, $\Z^*(G)\le \frac n 2$ for every connected graph of order $n\ge 2$.  This bound is achieved by $K_n$, $K_{p,q}$ with $p,q\ge 2$, and some other graphs. The next result uses forts of order two to characterize graphs with $\Z^*(G)= \frac n 2$.

\begin{proposition}\label{p:Z*n/2}
     Let $G$ be a connected graph of order $n\ge 2$.   Then every vertex of $G$ is in a fort of cardinality two if and only if $\Z^*(G)=\frac n 2$.  
\end{proposition}
\bpf
If every vertex of $G$ is in a fort of cardinality two, then  $\omega(v)=\frac 1 2$ for every vertex $v$ is an optimal  weight function by Lemma \ref{l:F2}.  If $z$ is not in any fort of cardinality two, then $\omega(v)=\frac 1 2$ for every vertex  $v\ne z$ and  $\omega(z)=0$ is a valid weight function, so $\Z^*(G)<\frac n 2$.
\epf

\begin{proposition}\label{p:ft:n/2}
Let $G$ be a graph of order $n$ with no isolated vertices.  Then $\ft(G)=\frac n 2$ if and only if   all of the following conditions are true:
\ben[$(1)$]
\item\label{c1:ft-half} $n=2k$ is even,
\item\label{c2:ft-half} there exists a partition of the vertices of $G$ into $F_i=\{x_i,y_i\}, i=1,\dots, k$,  and
\item\label{c3:ft-half} 
  for each $i\ne  j$, $S_{i,j}\cap E(G)=S_{i,j}$ or $S_{i,j}\cap E(G)=\emptyset$, where $S_{i,j}=\{x_ix_j, x_iy_j, y_ix_j, y_iy_j\}$.
\een
\end{proposition}
\bpf Let $G$ be a graph satisfying conditions \eqref{c1:ft-half}--\eqref{c3:ft-half}.  We show that each $F_i$ is a fort:  Let $u\not \in F_i$.  Then there exists $j\ne i$ such that $u\in F_j$. Moreover, $|N(u)\cap F_i|= 2$ if $S_{i,j}\cap E(G)=S_{i,j}$ and $|N(u)\cap F_i|= 0$ if $S_{i,j}\cap E(G)=\emptyset$.  So $F_i$ is a fort, and therefore $\ft(G)=\frac n 2$.

Now assume $G$ is a graph such that $\ft(G)=\frac n 2$.
Then $n=2k$ is even and there are disjoint forts  $F_i=\{x_i,y_i\}, i=1,\dots, k$ that partition  the vertices of $G$.  For $i\ne j$, define $S_{i,j}=\{x_ix_j, x_iy_j, y_ix_j, y_iy_j\}$.
 If $S_{i,j}\cap E(G)=\emptyset$, then there is nothing to prove.  So assume $S_{i,j}\cap E(G)\ne \emptyset$, and without loss of generality, $x_ix_j\in E(G)$. Then $|N(x_i)\cap F_j|\ne 1$ implies $x_iy_j\in E(G)$ and $|N(x_j)\cap F_i|\ne 1$ implies $x_jy_i\in E(G)$.  Then  $\abs{N(y_i)\cap F_j}\ne 1$ implies $ y_iy_j\in E(G)$.  Thus  $S_{i,j}\cap E(G)=S_{i,j}$. 
 \epf

The graphs $K_{2k}$ and $K_{2p,2q}$ are easy examples of such graphs,  as is a graph constructed from
the $r$th \emph{friendship graph}  $F_r=rK_2\vee K_1$, as seen in the next example.  

\begin{example} \label{ex:sK2veeK2} The graph $G=F_{r} \vee K_1 = rK_2\vee K_2$     is an example of a graph satisfying the conditions in Proposition \ref{p:ft:n/2}: Each $K_2$ is a fort, $S_{i,j}\cap E(G)=\emptyset$ when both $F_i$ and $F_j$ come from $rK_2$ and $S_{i,j}\cap E(G)=S_{i,j}$ when one of $F_i$ and $F_j$ is the $K_2$ that joins to $rK_2$.   Note that while $\ft(G)=r+1=\Z^*(G)$, $\Z(G)=r+2$ (since it is well known that $\Z(F_r)=r+1$ and adding a  vertex adjacent to every other vertex (a {universal vertex}) raises the zero forcing number by one \cite[Theorem 9.5, Proposition 9.16]{HogLinShad}).
\end{example}

Proposition \ref{p:ft:n/2} does not cover all graphs $G$ of order $n$ such that $\Z^*(G)=\frac n 2$.  There are odd order graphs with this property, including $K_{2k+1}$ and the graphs in the next example.

\begin{example}
    Let $G$ be a graph constructed from $K_n$ with $n\ge 4$ by deleting a disjoint set of edges $E_0$ such that either every vertex is an endpoint of an edge in $E_0$ or at least two vertices are not endpoints of edges in $E_0$.  Then every vertex is in a 2-element fort because the endpoints of an edge deleted are a fort and any two vertices that are not endpoints of a deleted edge are a fort.  Thus $\Z^*(G)=\frac n 2$.
\end{example}




\section{Bounds on the zero forcing number of the Cartesian product}\label{sec:cart_prod_bounds}
In this section, we apply results on fractional zero forcing number to establish lower bounds on the zero forcing number of the Cartesian product of two graphs.   

    For two hypergraphs, $H_1, H_2$, we define $H_1 \times H_2$ to be the hypergraph with \[ V(H_1 \times H_2) = V(H_1) \times V(H_2) \text{ and }  E(H_1 \times H_2) = \{e_1 \times e_2 ~|~ e_1 \in E(H_1), e_2 \in E(H_2)\}.\]

\begin{proposition}\label{l:fort-prod}\label{prop:cart_tens_prod}
Let $G$ and $G'$ be graphs, let $F_{G}$ be a fort of $G$, and let $F_{G'}$ be a fort of $G'$. Then   $F_{G}\times F_{G'}$ is a fort of $G \Box G'$.  
\end{proposition}

\begin{proof} 
  Note that $x=(v,v') \in V(G \Box G') \setminus \lp F_G\x F_{G'}\rp $ implies  $v \not \in F_{G}$ or $v' \not \in F_{G'}$.  
Observe first that if both $v \not \in F_{G}$ and $v' \not \in F_{G'}$, then $|N_{G\Box G'}(x) \cap \lp F_G\x F_{G'}\rp|=0$.
So without loss of generality suppose $v \not \in F_{G}$ and $v' \in F_{G'}$.  Let $\abs{N_{G}(v) \cap F_{G}}=k $, so  $k \neq 1$.  
For each $w \in F_G$, $(w,v') \in  F_G\x F_{G'}$ is a neighbor of $x=(v,v')$   if and only if $w$ is a neighbor of $v$ in $G$.  
So, it follows that $\left\vert N_{G \Box G'}(x) \cap \lp F_G\x F_{G'}\rp \right\vert = k\ne 1$.  
Thus, 
$F_G \times F_{G'}$ is a fort of $G \Box G'$. 
\end{proof}

 The next result is immediate from the previous proposition.
\begin{corollary}\label{cor:fortcart}  For graphs $G$ and $G'$,
\[\tau(\F_{G} \times \F_{G'}) \le \Z(G \Box G'),
\tau^*(\F_{G} \times \F_{G'}) \le \Z^*(G \Box G'),\mbox{ and }
\mu(\F_{G} \times \F_{G'}) \le \ft(G \Box G').\]
\end{corollary}


\begin{theorem}  {\rm \cite[Theorem 15 (Chapter 3)]{BergeHypergraphs}} \label{thm.berge}\label{sandwichproduct}
For two hypergraphs, $H_1$ and $H_2$,
\[ \mu(H_1)\mu(H_2) \le \mu(H_1 \times H_2) \le\tau^*(H_1) \mu(H_2) \le \tau^*(H_1) \tau^*(H_2)\] \[ =  \tau^*(H_1 \times H_2) \le  \tau^*(H_1) \tau(H_2) \le \tau(H_1 \times H_2) \le \tau(H_1) \tau(H_2) \]
\end{theorem}

 Using the hypergraphs $ H_1= \mathcal F_{G_1}$ and $ H_2= \mathcal F_{G_2}$ and taking into account Observation \ref{obs.fteqmatch} and Theorem \ref{thm.fortz}, the following string of inequalities is given by Theorem \ref{thm.berge}.

\begin{theorem}\label{bergetoZ}
For two graphs, $G_1$ and $G_2$,
\[ \ft(G_1)\ft(G_2) \le \mu(\mathcal F_{G_1} \times \mathcal F_{G_2}) \le\Z^*(G_1) \ft(G_2) \le \Z^*(G_1) \Z^*(G_2)\] \[ =  \tau^*(\mathcal F_{G_1} \times \mathcal F_{G_2}) \le  \Z^*(G_1) \Z(G_2) \le \tau(\mathcal F_{G_1} \times \mathcal F_{G_2}) \le \Z(G_1) \Z(G_2),\]
where $\mathcal F_{G_1}$ and $\mathcal F_{G_2}$ are the hypergraphs of minimal forts of $G_1$ and $G_2$ respectively.
\end{theorem}

The following corollary is an immediate consequence of Theorem \ref{bergetoZ}. 

\begin{corollary}\label{carttocomp}
    Let $G_1$ and $G_2$ be graphs.  
    \begin{enumerate}[$1.$]
        \item If $\mu(\mathcal F_{G_1} \times \mathcal F_{G_2})=\tau^*(\mathcal F_{G_1} \times \mathcal F_{G_2})$, then $\ft(G_2)=\Z^*(G_2)$.
        \item If $\mu(\mathcal F_{G_1} \times \mathcal F_{G_2})=\tau(\mathcal F_{G_1} \times \mathcal F_{G_2})$, then $\ft(G_2)=\Z^*(G_2)=\Z(G_2)$.
    \end{enumerate}
\end{corollary}



 Since $H_1\x H_2\cong H_2\x H_1$, the roles of $G_1$ and $G_2$ can be reversed in Theorems \ref{thm.berge} and \ref{bergetoZ}.  Thus Corollary \ref{carttocomp} remains true with $G_2$ replaced by $G_1$  in the conclusions.
The next result is immediate from  Theorem \ref{bergetoZ}  
 and  Corollary \ref{cor:fortcart}, 
since  $\tau^*(\F_G)=\Z^*(G)$ and $\tau(\F_G)=\Z(G)$ for every graph $G$.
\begin{corollary}\label{cor:cart_prod}
For two graphs $G$ and $G'$,
\[ \Z^*(G)\Z(G') \le \Z(G \Box G'). \]
\end{corollary}
Note that Corollary~\ref{cor:cart_prod} implies that if $\Z^{*}(G)=\Z(G)$, then
\[
 \Z(G)\Z(G')\le \Z(G\Box G').
\]
Families of graphs where $\Z^{*}(G)=\Z(G)$ are discussed in Section~\ref{subsec:zstar_eq_z}.
For now, we note that we can strengthen the bound given above  by using the next lemma. 

\begin{lemma} \label{lem.crossfortplusone}
 Let $G$ and $G'$ be graphs each containing an edge. Then $\Z(G \Box G') \geq \tau(\F_G \x \F_{G'})+1$.
\end{lemma}

 \begin{proof}   Suppose $B$ is a minimum transversal of $\F_G \x \F_{G'}$, i.e., $B$  is a subset of $V(G \Box G')$ of minimum size such that  $B \cap (F \times F') \neq \emptyset$ for each  $F \in \mathcal F_G$ and $F' \in \mathcal F_{G'}$.
Let $C_G$ be a nontrivial component of $G$, $C_{G'}$ be a nontrivial component of $G'$, and  let $C=C_G\cp C_{G'}$, which is a component of $G \cp G'$.  Now let $(g,g') \in B \cap V(C)$.  We will show that there exists $x \in N_G(g)$ such that $(x,g') \not \in B$.  Once this is done, an identical argument will show that there exists $x' \in N_{G'}(g')$ such that $(g,x') \not \in B$, and thus that every vertex in $B \cap V(C)$ has two neighbors in $V(C) \setminus B$.  Since $B$ is an arbitrary minimum transversal of $\mathcal F_G \times \mathcal F_{G'}$ and $V(C)\setminus B$ is a fort of $G \Box G'$, this completes the proof.  

 Let $\{g_i\}_{i=1}^m$ be an enumeration of $N_G(g)$ and let $g_0=g$; note that $m\ge 1$ since $C_G$ is nontrivial.  If $(g_1,g') \not \in B$, then we are done.  Otherwise, since $B$ is of minimum size there exists a fort $F_1$ of $G$  and a fort $F'$ of $G'$ such that  $(g_1,g') \in F_1\x F'$ and $(g_0,g') \not \in F_1\x F'$, i.e.,  $g_1 \in F_1$ and $g_0 \not \in F_1$.  Let $k$ be such that $1 \leq k \leq m-2$ and suppose that $\{(g_i,g')\}_{i=1}^k \subseteq B$ and there exists a set of forts $\{F_i\}_{i=1}^k$ of $G$ such that for each $F_i$ we have $g_i \in F_i$ and $g_j\not\in  F_i $ for $j=0,\dots,i-1$.  If $(g_{k+1},g') \not \in B$, then we are done.  Otherwise, since $B$ is of minimum size there exists a fort $F_{k+1}$ of $G$ such that $g_{k+1} \in F_{k+1}$ and $g_i\not\in  F_{k+1} $ for $i=0,\dots,k$.  So we may suppose that $\{(g_i,g')\}_{i=1}^{m-1} \subseteq B$ and there exists a set of forts $\{F_i\}_{i=1}^{m-1}$ of $G$ such that for each $F_i$ we have $g_i \in F_i$ and $\{g_j\}_{j=0}^{i-1} \cap F_i = \emptyset$.  If $ (g_m,g') \not \in B$, then we are done.  So suppose, by way of contradiction, that $(g_m,g') \in B$.  Then since $B$ is of minimum size it follows that there exists a fort $F_m$ of $G$ such that $g_m \in F_m$ and $g_i\not\in  F_{m} $ for $i=0,\dots,m-1$.  This is a contradiction because  it implies that $\abs{N_G(g) \cap F_m} =1$.  
Thus it follows that at least one member of $\{(g_i,g')\}_{i=1}^m$ is a member of $V(C) \setminus B$.
\end{proof}

The bound in Lemma \ref{lem.crossfortplusone} is sharp as seen by the next example.

\begin{example}
Consider the infinite class of graphs $G=K_r \Box P_s$, for $r, s \geq 2$.  Since each pair of vertices in $K_r$ is a fort, if $B \cap F\ne \emptyset$ for each fort $F \in \mathcal F_{K_r} \times \mathcal F_{P_s}$, then $\abs{B} \geq r-1$ and so
\[r \leq \abs{B}+1 \leq \Z(K_r \Box P_s)=r,\]
 where the last equality was established  in \cite{aim}. \end{example}

The bound in Lemma \ref{lem.crossfortplusone} does not always have equality, and, as is witnessed by the next example, the gap grows arbitrarily large.
\begin{example}
 Consider $P_r \Box P_r$, for $r \geq 3$.  Since every fort of $P_r$ contains the endpoints of $P_r$, for $V(P_r)=\{1,2,\dots,r\}$ enumerated in path order, choosing $B=\{(1,1)\}$ provides a  transversal of $\mathcal F_{P_r} \times \mathcal F_{P_r}$.  Thus, $ \tau(\F_{P_r} \Box \F_{P_r})+1=2$, but  $3 \leq r=\Z(P_r \Box P_r)$, where the last equality was established  in \cite{aim}. 
\end{example}



\begin{theorem}
\label{thm.zequalzstartrue}
Let $G$ and $G'$ be graphs each containing an edge, and suppose $\Z(G') = \Z^*(G')$. Then
\[ \Z(G \Box G') \ge \Z(G)\Z(G') + 1.\]
In particular, Conjecture \ref{con:main} holds whenever one of the two graphs has its zero forcing number equal to its fractional zero forcing number. 
\end{theorem}

\begin{proof}  
 Let $\F_G$ and $\F_{G'}$ be the hypergraphs of forts for $G$ and $G'$ respectively. Since $G$ and $G'$ each contain nontrivial components,   from  Lemma \ref{lem.crossfortplusone}  and   Theorem \ref{bergetoZ}, it follows that
\[\Z(G \Box G') \geq \tau(\F_G \times \F_{G'})+1 \geq \tau(\F_G)\tau^*(\F_{G'})+1=\Z(G)\Z^*({G'})+1=\Z(G)\Z({G'})+1. \qedhere\]
\end{proof}

  Table \ref{table2} in the next section lists families of graphs $G'$ for which $\Z(G') = \Z^*(G')$, to which Theorem \ref{thm.zequalzstartrue} applies.
The following corollary is an immediate consequence of Theorem \ref{thm.zequalzstartrue} and Corollary \ref{carttocomp}.
\begin{corollary}\label{carttocomptocart}
    If there exist graphs $G_1$ and $G_2$, with $G_2$ containing an edge, such that $\mu(\mathcal F_{G_1} \times \mathcal F_{G_2})=\tau(\mathcal F_{G_1} \times \mathcal F_{G_2})$, then for any graph $G_3$ containing an edge, $\Z(G_2 \Box G_3) \geq \Z(G_2)\Z(G_3)+1$.
\end{corollary}

\subsection{Graphs for which $\Z^{*}(G) = \Z(G)$}\label{subsec:zstar_eq_z}

As shown in Examples \ref{ex:path}, \ref{ex:cycle},  and \ref{ex:double_foliation}, paths, even cycles, and double foliations $H\circ 2K_1$ all satisfy $\Z^{*}(G) = \Z(G)$.  We can generalize the latter two examples; to do so we need to recall some definitions.   The \emph{subdivision} of  edge $e=uv$ in a graph $G$, denoted by $G_e$, 
is the graph obtained from $G$ by adding a new vertex $w$, adding two edges $uw$ and $vw$ and deleting the edge $e$.

\begin{example}  \label{ex:double-leaf-subdiv}
    Let $G'$ be a graph of order $r\ge 2$ with $V(G')=\{u_1,\dots,u_r\}$.  Denote the two leaf vertices adjacent to $u_k$  in $G'\circ 2K_1$ by $x_k$ and $y_k$.    Construct $G$ from $G'\circ 2K_1$ by subdividing edge $x_ku_k$ as many times as desired (including none) and denote the resulting additional vertices (if any) in path order by $x_k^1,\dots,x_k^{i_k}$, and similarly subdividing edge $y_ku_k$, for $k=1,\dots,r$.  
    Since $\{x_1,\dots,x_r\}$ is a zero forcing set  of $G$, $\Z(G)\le r$.  Since we have disjoint forts $F_k=\{x_k,x_k^1,\dots,x_k^{i_k},y_k,y_k^1,\dots,y_k^{j_k}\}$ for $k=1,\dots,r$, $\ft(G)\ge r$ and $\Z(G)=\Z^*(G)=\ft(G)$.  Let $\LL$ denote the family of graphs just defined. 
\end{example}

We now give a characterization for when trees have $\Z(G) = \Z^*(G)$.
Recall a \emph{path cover} of a graph $G$ is a  set of vertex-disjoint paths of $G$ such that every vertex is in one of the paths and each path is an induced path in $G$; the \emph{path cover number} $\PC(G)$ is the minimum  size of  a  path-cover of $G$. It is well-known that $\Z(G)\ge \PC(G)$ and $\PC(G)\ge\frac {\ell(G)} 2$ where $\ell(G)$ is the number of leaves in $G$. For trees, however, $\Z(T) = \PC(T)$. 

Let us define $\T$ to be the family of trees such that that there is a minimum path cover $\PP$ that satisfies the following two conditions:
 \ben[(a)]
 \item\label{LLconda}  every path in $\PP$ covers two leaves of $T$ 
 \item\label{LLcondb}  for any three paths $P_{(1)}$, $P_{(2)}, P_{(3)} \in \PP$, there are no adjacent vertices $u, v \in V(P_{(1)})$ such that $u$ is adjacent to a vertex in $P_{(2)}$ and $v$ is adjacent to a vertex in $P_{(3)}$. 
 \een


\begin{lemma} \label{lem:twoleaves}
For a tree $T$, every fort contains two leaf vertices.
\end{lemma}

 \begin{proof}
    It suffices to prove the following statement: Let $T$ be a tree. Then, every set of all but at most one leaf of $T$ is a zero forcing set of $T$. 
    We proceed via strong induction on the order $n$ of the tree. 
    The base case, $n=2$, is clear since the only fort of $P_{2}$ is $V(P_{2})$.
    For the induction step, let $n\geq 2$, and assume that the result holds for all trees of order between $2$ and $n$.
    Let $T$ be a tree of order $n+1$ and let $S$ be a set of all but at most one leaf of $T$.
    Suppose that each vertex in $S$ is filled and each vertex in $V(T)\setminus{S}$ is not filled. 
    Since each vertex in $S$ is a leaf, they each force a vertex in $V(T)\setminus{S}$.
    After these forcings are applied, $T-S$ is a tree of order between $2$ and $n$ such that all but at most one leaf is filled. 
    Hence, the induction hypothesis implies that all vertices in $T-S$ can be forced and it follows that $S$ is a zero forcing set of $T$. 
\end{proof}

\begin{lemma} For any tree $T$, $\Z^{*}(T)\le\frac{\ell(T)}{2}\leq\Z(T)$.
\end{lemma}
\bpf
By Lemma \ref{lem:twoleaves} every fort of $T$ must contain at least two leaves, so    
 $\Z^*(T) \le \frac{\ell(T)}{2}$. 
  Since each path in a path cover covers at most two leaves and  $Z(T) = \PC(T)$, it follows that  $\frac{\ell(T)}{2}\le\Z(T)$.
\epf
\begin{theorem} \label{thm:LL}
For any tree $T$,
$\Z^*(T) = \Z(T)$ if and only if $T \in \T$, in which case, $\ft(T) = \Z^*(T) = \Z(T)=\frac{\ell(T)}2.$
\end{theorem}

\bpf 
Suppose first that $T\in \T$ and let $\PP=\{P_{(1)},\dots,P_{\left(\frac {\ell(T)} 2\right)}\}$ be a path cover with the specified properties.  Then $\Z(T)=\PC(T)=\frac {\ell(T)} 2$.  Let $F_{(k)}$ be the vertices of $P_{(k)}$ of degree at most two.  Then $F_{(k)}$ is a fort by condition \eqref{LLcondb}, so $\ft(T)\ge \frac{\ell(T)} 2$.
Thus, $\Z(T)=\Z^*(T)= \ft(T) = \frac{\ell(T)}{2} $.

It remains to show that 
$T \not \in \T$ implies $\Z^*(T) < \Z(T)$.
 If there is a minimum path cover of $T$ that fails condition \eqref{LLconda}, then $\PC(T) > \frac{{\ell(T)}}{2}$, so $\Z(T) > \frac{\ell(T)}{2}\geq \Z^{*}(T)$ and we are done. So we  assume that every minimum path cover of $T$ satisfies 
condition \eqref{LLconda} but fails condition \eqref{LLcondb}.
Let $\PP=\{P_{(1)},\dots,P_{\left(\frac {\ell(T)} 2\right)}\}$ be a minimum path cover. Without loss of generality, let 
$P_{(1)} \in \PP$ be a path that has two adjacent vertices adjacent to other paths in $\PP$ (i.e., $P_{(1)}$ violates condition \eqref{LLcondb}).

Next we show that every fort of $T$ contains two leaves belonging to the same path of $\PP$.
Note that for a tree $T$,  a minimum zero forcing set can be constructed by arbitrarily choosing either end of the minimum path cover (see   \cite
{aim}). By the contrapositive, if a set of vertices is not a zero forcing set, at least one path in $\PP$ must have both its leaves unfilled. Hence, every fort of $T$ contains two leaves from the same path in $\PP$.

Finally, we show that any fort containing the two leaves of $P_{(1)}$ contains another leaf.
Suppose for a contradiction that all other leaves are filled. Then, all those vertices will force until they reach $P_{(1)}$ by forcing two adjacent vertices of $P_{(1)}$. Those two adjacent vertices will force the rest of $P_{(1)}$ and hence $T$. Therefore, there can be no fort of $T$ that excludes all leaves not contained in $P_{(1)}$.   

Together, we have that every fort of $T$ either
contains two leaves of a path that is not $P_{(1)}$ or it contains both leaves of $P_{(1)}$ and at least one additional leaf. Consider the  following weighting on the vertices: $x_v = 0$ if $v$ is not a leaf, $x_v = \frac14$ if $v$ is a leaf in $P_{(1)}$, and $x_v = \frac12$ if $v$ is a leaf not in $P_{(1)}$. 
For either case for forts, this  weighting provides that each fort has a weight sum of at least 1 with a total weight over all vertices of $\frac{\ell(T)}{2} - \frac{1}{2} < \frac{\ell(T)}{2}$.
\epf

 A graph $G$ is  a \emph{graph of two parallel paths} if $\PC(G)=2$ and the graph can be drawn in the plane in such a way that the paths are
parallel line segments, the edges  between the two paths do not cross, and each edge is drawn as a straight line segment.  A graph that consists of two connected components, each of which is a path, is a graph of two parallel paths, but a single path is not. A \emph{polygonal path} is a graph that can be constructed by starting with a cycle and adding one cycle at a time by identifying an edge of a new cycle with an edge of the most recently added cycle that has a vertex of degree 2.  Every polygonal path is a graph on two parallel paths but not conversely.  The next example describes a family of graphs $G$ on two parallel paths that have $\Z(G)=\Z^*(G)$, including polygonal paths of even order.

\begin{example}\label{ex:2parpath-evencycle} It is well-known that a graph $G$ has $\Z(G)=2$ if and only if $G$  is a graph on two parallel paths \cite{DDR},  \cite[Theorem 9.12]{HogLinShad}. 
The \emph{maximum cycle length} of a graph $G$ that  is not a forest is the length of the longest cycle in $G$.  
For a graph $G$ that is a graph on two parallel paths that has a cycle, $G$ has a unique cycle of maximum length; denote the vertices of this  cycle by $C(G)$. Let $\PAR$ be the set of graphs $G$ such that $G$ is a graph of two parallel paths and  $|C(G)|$ is even. We can see that every graph $G\in \PAR$ has $\ft(G)=2$, which implies that $\Z(G)=\Z^*(G)$:
Denote the vertices of $C(G)$ by $u_1,\dots,u_{2k}$. Observe that $G$ can be obtained from $G[C(G)]$ by adding zero, one, two, three, or four pendent paths,  each from a distinct vertex of $C(G)$ (additional restrictions are needed to ensure $G$ is a graph on two parallel paths); 
 Examples 6, 7, and 8 in \cite{proptime} show graphs on two parallel paths that have cycle(s). Define $F_e$ to be the set consisting of  $W_e=\{u_2,u_4,\dots,u_{2k}\}$ and the vertices of any path pendent from a vertex in $W_e$, and define $F_o$ similarly.  Then $F_e$ and $F_o$ are disjoint forts.  
\end{example}

 The next table lists  families of graphs $G$ having $\Z(G)=\Z^*(G)$.  Note there is some duplication (e.g. an even cycle is an even polygonal path, which is also in $\PAR$), but the results are displayed this way for ease of use.

\begin{table}[h!]
\renewcommand{\arraystretch}{1.3}
\begin{center}
\noindent \begin{tabular}{| l | c | c| c| }
\hline
result \#&  $G$ & order &  $\Z^*(G)=Z(G)$\\
\hline
\ref{ex:path} & $P_n$  & $n$& $1$  \\
\hline
\ref{ex:cycle} & $C_{2k}$   & $2k$& $2$  \\
 \hline
\ref{ex:double_foliation} & $G'\circ 2K_1$  & $3|V(G')|$& $|V(G')|$  \\
 \hline
\ref{ex:double-leaf-subdiv} & $G\in\LL$  & & $\frac{\ell(G)} 2$   \\
 \hline
\ref{thm:LL} & $T\in\T$  & & $\frac{\ell(T)} 2$ 
\\
\hline
\ref{ex:2parpath-evencycle} & polygonal path, even order  & $2k$ & $2$  \\
\hline
\ref{ex:2parpath-evencycle} & $G\in\PAR$  &  & $2$  \\\hline
\end{tabular}
\caption{Graphs $G$ for which $\Z(G)=\Z^*(G)$. }\label{table2}
\end{center}
\end{table}

 It is worth noting that in each example above where $\Z^{*}(G)=\Z(G)$ it is also the case that $\ft(G)=\Z(G)$.  In particular, $\Z^{*}(T)=\Z(T)$ implies $\ft(T)=\Z(T)$ whenever $T$ is a tree.  This  motivates the following question.
\begin{question}\label{fort_question}
Is it the case that $\Z^*(G) = \Z(G)$ if and only if $\ft(G)=\Z(G)$ for all graphs $G$?
\end{question}

By Theorem \ref{thm:LL}, Question \ref{fort_question} is  answered in the affirmative for trees.
For  arbitrary graphs $G$, one direction of Question~\ref{fort_question} follows from the inequality in~\eqref{eq:ftz-ineq}, that is, if $\ft(G)=\Z(G)$ then $\Z^{*}(G)=\Z(G)$.
For the other direction, one might try to show that $\Z^{*}(G) \leq \frac{1}{2}\left(\ft(G) + \Z(G)\right)$. 
If all minimal forts of $G$ have size $2$ (see Examples~\ref{ex:double_foliation} and~\ref{ex:triangles}), then it is true that $\Z^{*}(G) \leq \frac{1}{2}\left(\ft(G) + \Z(G)\right)$ 
because in this case the hypergraph of minimal forts $\F_{G}$ is a simple graph and Theorem 3 in Chapter 3 of~\cite{BergeHypergraphs} implies that
\[
\tau^{*}(\F_{G}) \leq \frac{1}{2}\left(\mu(\F_{G}) + \tau(\F_{G})\right)
\]
and the result follows from Theorem~\ref{thm.fortz} and Observation~\ref{obs.fteqmatch}.  Unfortunately, $\Z^{*}(G) \leq \frac{1}{2}\left(\ft(G) + \Z(G)\right)$ is not true for all $G$.
Indeed, Example~\ref{ex:cycle} demonstrates that $\Z^{*}(C_{2k+1}) = \frac{2k+1}{k+1}$ and $\ft(C_{2k+1}) = 1$. 
Thus, 
\[
\Z^{*}(C_{2k+1})=\frac{2k+1}{k+1} > \frac{3}{2}=\frac{1}{2}(\ft(C_{2k+1})+\Z(C_{2k+1})),
\]
for $k\geq 2$.


\subsection{Bounds on $\ft(G\Box G')$ and $\ft^*(G\Box G')=\Z^*(G\Box G')$}\label{subsec:fort-cart-prod-bds}
 
  In this section we establish Vizing-like bounds for the fort number and fractional zero forcing  number of Cartesian products.
 
 \begin{proposition}
\label{prop:cart_tens_prod_fractional}
 For all graphs $G$ and $G'$,
\[  \Z^{*}(G)\, \Z^*(G') \le \Z^{*}(G \Box G').\]
\end{proposition}
\begin{proof}
By Corollary \ref{cor:fortcart} we have  
$ \tau^{*}(\F_{G} \times \F_{G'}) \le \Z^{*}(G \Box G').$
On the other hand, from Theorem \ref{thm.berge}, 
$\tau^*(\F_{G}) \tau^*(\F_{G'})=\tau^{*}(\F_{G} \times \F_{G'}),$ which completes the proof of the inequality because $\tau^*(\F_G)=\Z^*(G)$. 
\end{proof}

 The next proposition 
 shows the bound is sharp.  Define $G_m=K_{2m} \Box K_{2m}$ for $m \geq 2$.
 
 \begin{proposition}
    \label{ex-sharp-bound}
   For $m\ge 2$, any fort of $G_m$ contains at least $4$ vertices, and \[\Z^*(G_m)=\Z^*(K_{2m})\Z^*(K_{2m})= m^2.\]
    \end{proposition}
    \bpf
First note that $\ft(K_{2m})=m$  (see Example \ref{ex:complete}). We  show that 
any fort of $G_m$ contains at least 4 vertices, which implies 
\[m^2=\Z^*(K_{2m}) \Z^*(K_{2m})\le \Z^*(G_m)\leq \frac{|V(\mathcal F_{G_m})|}{4} \le \frac{|V(G_m)|}{4}=m^2,\] 
establishing the equality.  

Let $V(G_m)=\{(i,j):1 \leq i,j \leq 2m\}$ and let  $F$ be a fort of $G_m$.  Suppose first  that for some $i_0 \in \{1,2,\dots, 2m\}$, $F \cap \{(i_0,j)\}_{j=1}^{2m}=\{(i_0,j_0)\}$.  Then for each vertex $(i_0,j_1)$ with $j_1 \neq j_0$, $(i_0,j_1) \not \in F$. 
 Since $N_{G_m}((i_0,j_1))=\{(i_0,j)\}_{j \neq j_1} \cup \{(i,j_1)\}_{i \neq i_0}$ and $\abs{\{(i_0,j)\}_{j \neq j_1} \cap F}=1$, it follows that $\{(i,j_1)\}_{i \neq i_0} \cap F \neq \emptyset$.  Since $j_1$ was chosen arbitrarily, for each $j \neq j_0$, $\{(i,j)\}_{i \neq i_0} \cap F \neq \emptyset$, and thus $\abs{F} \geq 2m \geq 4$.   By   symmetry, $\abs{F \cap \{(i,j_0)\}_{i=1}^{2m}}= 1$ implies $\abs{F} \geq  4$.
 
 Now suppose   $\abs{F \cap \{(i_0,j)\}_{j=1}^{2m}}\neq 1$ for each $i \in \{1,2,\dots,2m\}$ and  
  $\abs{F \cap \{(i,j_0)\}_{i=1}^{2m}}\neq 1$  for each $j_0 \in \{1,2,\dots,2m\}$.  Since $F \neq \emptyset$,
  for some $i_1,i_2,i_3,j_1,j_2 \in \{1,2,\dots,2m\}$, with $ i_2 \neq i_1, \ i_3 \neq i_1 $, and $j_1 \neq j_2$, $\{(i_1,j_1), (i_1,j_2), (i_2,j_1), (i_3,j_2)\} \subseteq F$, and thus $\abs{F} \geq 4$.  So for any fort $F$ of $G_m$, we have that  $\abs{F} \geq 4$. 
\epf

 The gap in the bound in Proposition \ref{prop:cart_tens_prod_fractional}
can grow arbitrarily large, as seen in the next example.  
 \begin{example}\label{Vbd-ft-gap}
Consider $G=C_m \Box C_m$ with $m\ge 5$. Enumerate the vertices of each cycle by $0,1,\dots,m-1$ in cyclic order (so $ V(G)=\{(i,j):0 \leq i,j \leq m-1\}$), and perform arithmetic modulo $m$. For each $k$ with $0 \leq k \leq m-1$, define the \emph{$k$th diagonal} $D_k=\{(i,i+k):i=0,\dots,m-1\}$.  We see that each $D_k$ is a fort of $G$:  Fix $k$ and $i$ and consider the vertex $(i,j)$. For $j=i+k$, $(i,i+k)\in D_k$.  For $j=i+k+1$, $(i,i+k+1)\sim (i,i+k)$ and $(i,i+k+1)\sim (i+1,i+k+1)$, and $j=i+k-1$ is similar. For $j\ne i+k,i+k+1,i+k-1$, $N[(i,j)]\cap D_k=\emptyset$, so $D_k$ is a fort. 
For $k_1 \neq k_2$, $D_{k_1} \cap D_{k_2} =\emptyset$, so $\Z^*(G)\ge \ft(G) \geq m >4\ge \Z^*(C_m)\Z^*(C_m)$, with the last inequality following from  $\Z^*(C_m)\le\Z(C_m)=2$.  
\end{example}


Next we present a Vizing-like bound for fort number, which also follows from work of Anderson et al.~ in \cite{powdompart} as discussed below.
  \begin{corollary}
    \label{boxfort} 
Let $G$ and $G'$ be graphs.  Then $\ft(G) \cdot \ft(G') \leq \ft(G \Box G')$  and the bound is sharp.
\end{corollary}
\bpf  We have $\ft(G_1) \ft(G_2) = \mu(\F_{G_1})\mu(\F_{G_2}) \le \mu(\F_{G_1} \times \F_{G_2}) \le \ft(G_1 \Box G_2)$ where the first equality follows from Observation \ref{obs.fteqmatch}, the first inequality results from Theorem \ref{sandwichproduct} and the final inequality results from Proposition \ref{l:fort-prod}.  Proposition \ref{ex-sharp-bound} shows that the bound in Corollary~\ref{boxfort} is sharp  because 
$m^2=\ft(K_{2m})\ft(K_{2m}) \leq \ft(G_m) \leq \Z^*(G_m)=m^2$.

\epf

  Example \ref{Vbd-ft-gap} shows the gap can be arbitrarily large since $\ft(C_m)\le \Z^*(C_m)$.
 
 In \cite{powdompart}, Anderson et al.~proved results concerning failed zero forcing partitions and their values for Cartesian products of graphs.  These results can be interpreted in terms of the fort number, using 
 definitions  introduced in \cite{powdompart}.  Let $G$ be a graph and let $\Pi=\{\Pi_i\}_{i=1}^k$ be a partition of the vertices of $G$ such that for each $i \in [k]$, $V(G)-\Pi_i$ cannot force $G$.  Then $\Pi$ is a \emph{failed zero forcing partition of $G$}.   Define $z_G$ to be the maximum number of sets in a failed zero forcing partition of $G$, i.e.,
\[z_G=\max\{j: \abs{\Pi}=j \text{ and } \Pi \text{ is a failed zero forcing partition of } G\}.\]
The next result  is a corollary of Theorem \ref{thm.fortchar}.

\begin{corollary}\label{zgftg}
Let $G$ be a graph.  Then $z_G=\ft(G)$.
\end{corollary}
\begin{proof}
Let $\mathcal F=\{F_i\}_{i=1}^{\ft(G)}$ be a collection of disjoint forts of $G$ and let $\Pi=\{\Pi_i\}_{i=1}^{\ft(G)}$ be a partition of $V(G)$ such that for each $i \in [\ft(G)]$, $F_i \subseteq \Pi_i$.  Since for each $\Pi_i \in \Pi$, $\Pi_i$ contains a fort, by Theorem~\ref{thm.fortchar}, $V-\Pi_i$  cannot force $G$. Thus $\Pi$ is a failed zero forcing partition of $G$, and in particular $z_G \geq \ft(G)$.

Next let $\Pi=\{\Pi_i\}_{i=1}^{z_G}$ be a failed zero forcing partition of $G$.  Since for each $i \in [z_G]$, $V(G)-\Pi_i$ does not zero force $G$, by Theorem \ref{fortsiffzfs}, for each $i \in [z_G]$, $\Pi_i$ contains a fort $F_i$.  Since $\Pi$ is a partition of $V(G)$, it follows that for each distinct $i,j \in [z_G]$, $F_i \cap F_j = \emptyset$.   Thus, $\{F_i\}_{i=1}^{z_G}$ is a collection of disjoint forts of $G$, and so $\ft(G) \geq z_G$. \end{proof}

Thus the  next result is equivalent to Corollary \ref{boxfort}.
\begin{theorem}\label{boxzg}{\rm \cite{powdompart}}
Let $G$ and {$G'$} be graphs.  Then $z_G \cdot z_{G'} \leq z_{G \Box {G'}}$. 
\end{theorem}


\subsection{More Lower Bounds on $\Z(G\Box G')$}\label{subsec:more_bounds}

 In this section we present a variety of additional lower bounds on the zero forcing number of a Cartesian product of graphs.

\begin{theorem} \label{lower-sum} {\rm \cite
{bergesim1974hypergraphs}}
 Let $H_1$ and $H_2$ be two hypergraphs. Then $\tau(H_1 \times H_2)\ge \tau(H_1)+\tau(H_2)-1$. 
\end{theorem}

\begin{proposition} \label{zero-sum}
Let $G$ and ${G'}$ be  graphs each with an edge.
Then,
\[
\Z(G\Box {G'})\ge \Z(G)+\Z({G'})
\]
and this bound is sharp.
\end{proposition}

\begin{proof}
By Lemma \ref{lem.crossfortplusone} and Theorem \ref{lower-sum}, we have
\[
\Z(G\Box {G'})\ge \tau({\cal F}_G\times {\cal F}_{G'})+1\ge \tau({\cal F}_G)+\tau({\cal F}_{G'})=\Z(G)+\Z({G'}).
\]
To see that the bound is sharp, note that
\[\Z(K_r \Box P_n)=r=(r-1)+1=\Z(K_r)+\Z(P_n)
\]
where the equalities involving $\Z$ were established in~\cite{aim}. 
\end{proof}

The next example shows that the lower bound in Proposition \ref{zero-sum} is sometimes better than that in Theorem \ref{Thm1}.
\begin{example}
    Let $G=G'$ be pentasun. Then the lower bound for $\Z(G\Box {G'})$ in Proposition \ref{zero-sum} is equal to $2 Z(G)=6$, while using Theorem \ref{Thm1}, the lower bound is equal to $M^2(G)+1=5$. 
\end{example}

 As is common with sharp bounds on graph parameters (e.g., $\M(G)\le\Z(G)$), there exist families of graphs for which the gap grows arbitrarily large.
\begin{example}
   For $s \geq 2$, we have that $\Z(P_s \Box P_s)=s \geq 2 = \Z(P_s)+\Z(P_s)$ where the equalities were established in \cite{aim}.
\end{example}

\begin{theorem}[Lovasz 1975 \cite{lovasz1975ratio}]\label{thm:lovasz}
Let $H$ be a hypergraph with maximum degree $\Delta$ (i.e., each vertex is in at most $\Delta$ edges).
Then,
\[ \frac {\tau(H)}{\tau^*(H)} \le 1 + \ln \Delta. \]
\end{theorem}

\begin{corollary}\label{cor:lovasz}
Let $G'$ be a graph containing an edge where every vertex in $G'$ is in at most $\Delta$ minimal forts of $G'$.
Then for every graph $G$ that contains an edge,
\[ 1+\frac{ \Z(G) \Z(G')}{1 + \ln \Delta} \le \Z (G \Box G') \]
\end{corollary}

\begin{proof}
By Lemma \ref{lem.crossfortplusone} and Theorems \ref{sandwichproduct} and \ref{thm:lovasz} 
we have 
\[
\Z(G \Box G') \ge \tau(\F_G \times \F_{G'})+1 \ge \tau(\F_G) \tau^*(\F_{G'})+1\ge \frac{\tau(\F_G) \tau(\F_{G'})}{1+\ln \Delta}+1.
\]
\end{proof}


As with many bounds, there exist families of graphs for which the gap grows arbitrarily large.
For instance, by~\cite
{Becker2024}, the path graph has an exponential number of minimal forts (in the order of the graph) and every fort of the path graph contains the two pendent vertices.  



\section{Graphs satisfying $\Z(G \Box G')=\Z(G)\Z(G')+1$}\label{s:realize-lowerbd}

 In this section we exhibit a family of graphs attaining the bound 
$\Z(G \Box G')=\Z(G)\Z(G')+1$.
We begin with several definitions.  A vertex $v$ of a graph $G$ is called a \emph{cut-vertex} if deleting $v$ and all edges incident to it increases the
number of connected components of $G$. A \emph{block} of a graph $G$ is a maximal connected induced subgraph of $G$
that has no cut-vertices. A graph is \emph{block-clique} (also called 1-chordal) if every block is a clique. 
  Note that a block-clique graph can be constructed iteratively as follows: $G_1$ is a clique.  Given that $G_{k-1}$ is block-clique, then  define $G=G_{k-1}\cup  G_k'$ where $G_k'$ is a clique and $V(G_{k-1})\cap V({G_k'})=\{v_k\}$ for some vertex $v_k$.  The next result is well-known and useful. 

\begin{theorem}\label{t:b-c}{\rm  \cite
{block-cliquethm}}
If $G$ is a block-clique graph, then $\Z(G)=\M(G)$.
\end{theorem}

 \begin{definition}\label{def:star-cliq path}
A \emph{star-clique path} is a graph $G$ that can be constructed as $G=G_k$  where $G_j=\cup_{i=1}^j G_i'$ for $j=1,\dots,k$, each $G_i'$ is a star or a clique (each with at least two vertices), and the following conditions are satisfied: 
\ben
\item  $V(G_{i-1}')\cap V(G_i')=\{v_i\}$ for some vertex $v_i$ for $i\in \{2,3, \ldots, k\}$ and $V(G_{j}')\cap V(G_i')=\emptyset$ for $j=1,\dots,i-2$.
\item  $v_i\not\in\{v_2,\dots,v_{i-1}\}$ for $i\in  \{3, \ldots, k\}$.
\item If ${G_i'}$ is a star for some $i\in \{2,3, \ldots, k-1\}$, then $\deg_{{G_i'}}(v_i)=1$ and  $\deg_{{G_i'}}(v_{i+1})=1$. 
 If ${G_1'}$ is a star then $\deg_{{G_1'}}(v_2)=1$, and if ${G_k'}$ is a star then $\deg_{{G_k'}}(v_k)=1$.
\een
\end{definition}
Figure \ref{SCP} shows examples of star-clique paths.
\begin{figure}[ht]
\centering
\begin{tikzpicture}[scale=0.75]
    \begin{scope}
    \node[circle,draw=black,fill=white] (0) at (0,0) {};
    \node[circle,draw=black,fill=white] (1) at (0,1) {};
    \node[circle,draw=black,fill=white] (2) at (0,2) {};
    \node[circle,draw=black,fill=white] (3) at (0,3) {};
    \node[circle,draw=black,fill=white] (4) at (1,3) {};
    \node[circle,draw=black,fill=white] (5) at (-1,3) {};
    \node[circle,draw=black,fill=white] (6) at (-2,3) {};
    \node[circle,draw=black,fill=white] (7) at (-1.5,3.86) {};
    \node[circle,draw=black,fill=white] (8) at (-1.5,2.14) {};

    \foreach \x/\y in {0/1,1/2,2/3,3/4,3/5,5/6,5/7,5/8,6/7,6/8,7/8}
        \draw[black,=>latex',-,very thick] (\x) -- (\y);
    \end{scope}
    \begin{scope}[xshift=150]
    \node[circle,draw=black,fill=white] (0) at (-0.5,0) {};
    \node[circle,draw=black,fill=white] (1) at (0.5,0) {};
    \node[circle,draw=black,fill=white] (2) at (0,0.86) {};
    \node[circle,draw=black,fill=white] (3) at (0,1.86) {};
    \node[circle,draw=black,fill=white] (4) at (0,2.86) {};
    \node[circle,draw=black,fill=white] (5) at (1,2.86) {};
    \node[circle,draw=black,fill=white] (6) at (-1,2.86) {};
    \node[circle,draw=black,fill=white] (7) at (0,3.86) {};
    \node[circle,draw=black,fill=white] (8) at (0,4.86) {};

    \foreach \x/\y in {0/2,1/2,2/3,3/4,4/5,4/6,4/7,7/8}
        \draw[black,=>latex',-,very thick] (\x) -- (\y);
    \end{scope}
    \begin{scope}[xshift=300]
    \node[circle,draw=black,fill=white] (0) at (-0.5,0) {};
    \node[circle,draw=black,fill=white] (1) at (0.5,0) {};
    \node[circle,draw=black,fill=white] (2) at (0,0.86) {};
    
    \node[circle,draw=black,fill=white] (3) at (0,1.86) {};
    \node[circle,draw=black,fill=white] (4) at (0.95,2.55) {};
    \node[circle,draw=black,fill=white] (5) at (0.58,3.66) {};
    \node[circle,draw=black,fill=white] (6) at (-0.58,3.66) {};
    \node[circle,draw=black,fill=white] (7) at (-0.95,2.55) {};

    \node[circle,draw=black,fill=white] (8) at (0.58,4.66) {};

    \foreach \x/\y in {0/2,1/2,2/3,3/4,3/5,3/6,3/7,4/5,4/6,4/7,5/6,5/7,6/7,5/8}
        \draw[black,=>latex',-,very thick] (\x) -- (\y);
    \end{scope}
\end{tikzpicture}
\caption{Examples of star-clique paths}
\label{SCP}
\end{figure}
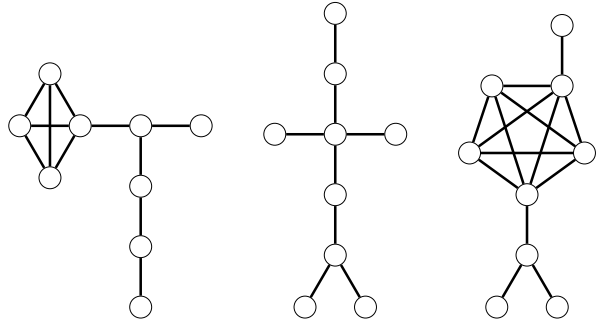
Since a star-clique path   is a block-clique graph, the next result follows from Theorem \ref{t:b-c}
\begin{corollary}\label{cor:star-cli-path}
 If $G$ is a star-clique path, then $\Z(G)=\M(G)$.
\end{corollary}

Next, we state an additional result that we use in the proof of  Proposition \ref{lem1}. 

\begin{lemma}{\rm \cite{DDR}} \label{lem0}
Let $G=(V_G, E_G)$ be a graph with cut-vertex $v\in V_G$. Let $W_1, \ldots, W_h$ be the vertex sets for the connected components of $G-v$ and for $1\le i \le h$, let $G_i=G[W_i\cup \{v\}]$. Then
\[
\Z(G)\ge \sum_{i=1}^h \Z(G_i)-h+1.
\]
\end{lemma}

 Note that for $r\ge 2$,  we have the following options for choosing a minimum zero forcing set $B'$ for  $K_r$  and $K_{1,r}$:  For $K_r$ and any vertices $u$ and $w$, we may choose $B'$ so that $u\in B'$ and $w\not\in B'$ (which implies $w$ does not perform a force when forcing $K_r$ with $B'$).  For $K_{1,r}$ and any leaves $u$ and $w$, we may choose $B'$ so that $u\in B'$ and $w\not\in B'$ (which implies $w$ does not perform a force when forcing $K_{1,r}$ with $B'$). Thus for a  star-clique path $G$ constructed from $G'_1, \dots, G'_k$ according to Definition \ref{def:star-cliq path}. we can choose zero forcing sets $B'_i$ for $G'_i$ for $i=1,\dots,k$  with the following properties:    For $i=2,\dots,k$, $v_i\in B'_i$ and $v_i\not\in B'_{i-1}$ and $v_i$ does not perform a force when forcing $G'_{i-1}$ with $B'_{i-1}$.  Such a zero forcing set of $G$ is called a \emph{canonical star-clique path set}.
\begin{proposition} \label{lem1}
Let $G$ be a star-clique path constructed from $G'_1, \dots, G'_k$ according to Definition \ref{def:star-cliq path}. Then  a canonical star-clique path set is a minimum zero forcing set of $G$ and 
\[
\Z(G)= \sum_{i=1}^k \Z(G'_i)-k+1.
\]
\end{proposition}
\begin{proof}
First we show by induction on $k$ that  $\Z(G)\ge  \sum_{i=1}^k \Z(G'_i)-k+1$.  This  is immediate when $k=1$. Assume the result is true for $G_{k-1}=\cup_{i=1}^{k-1}G'_i$.  
Applying Lemma \ref{lem0} to the cut-vertex $v_k$ of $G_k$ together with the induction assumption gives
 \[
 \Z(G_k)\ge  \Z(G_{k-1})+\Z(G'_k)-2+1\ge \lp \sum_{i=1}^{k-1} \Z(G'_i)-(k-1)+1\rp +\Z(G'_k)-1=\sum_{i=1}^k \Z(G'_i)-k+1.
 \]

 By definition, $B_k=B'_1\cup \bigcup_{i=2}^k \lp B'_i\setminus \{v_i\}\rp$ is a zero forcing set  of $G_k$. Therefore,  $\Z(G)\le |B_k|= \sum_{i=1}^k \Z(G'_i)-k+1$.   
\end{proof}


The  next theorem is the main result of this section and gives extremal graphs for Conjecture \ref{con:main}. 

\begin{theorem}\label{t:scp1}
 Let $G$ be a star-clique path and let $r\ge 2$. Then,
\[
\Z(K_r \Box G)=(r-1) \Z(G)+1  = \Z(K_r)\Z(G)+1.
\]
\end{theorem}
\begin{proof} Let $G$ be a star-clique path constructed from $ G'_1, \dots, G'_k$ according to Definition \ref{def:star-cliq path}.  For purposes of this proof, we view $K_{1,1}$ as a clique and $K_{1,2}$ as $G'_i=K_2$ and $G'_{i+1}=K_2$, so we assume a star is $K_{1,s}$ with $s\ge 3$.  Note first that $\Z(K_r \Box G)\ge (r-1) \Z(G)+1$ by Theorem  \ref{Thm1} and Corollary \ref{cor:star-cli-path}.   
 Choose a canonical star-clique path set $\hat B$ for $G$. We describe vertices by the vertex label in $G$ and the copy number $j=1,\dots,r$. In $K_r \Box G$, we define a  set  $B$ consisting of the same canonical star-clique path set $\hat B$ in the  $j$th copy of $G$ for $1\le j \le  r-1$, together with one vertex  $u$ in the $r$th copy of $G_1'$ such that  $u\ne v_2$  and $u$ is not  the  center of $G_1'$ if $G_1'$ is a star. 
 We show that the vertex set $B$ defined is a zero forcing set: 
 
 If $G'_1$ is a clique, then $v_2$ can be forced by $u$ in the $j$th of the copy of $G'_1$ for $j=1,\dots,r-1$.  Then each of the  unfilled vertices $w\ne v_2$ in the $r$th copy of $G'_1$ can be forced by its neighbor $w$ in another (fully filled) copy of $G'_1$.  Then $v_2$ in copy $r$ can be forced by $u$ in copy $r$.
 If $G'_1$ is a star, then the center vertex $c$ can be forced by $u$ in the $j$th copy of $G'_1$ for $j=1,\dots,r$.  Then  each unfilled vertex $w\ne v_2$ in the $r$th copy of $G'_1$ can be forced by its neighbor $w$ in another  copy of $G'_1$. Finally,  $v_2$ can be forced by $c$ in  the $j$th copy of $G'_1$ for $j=1,\dots,r$.

Now the $j$th copy of $G'_2$ contains a zero forcing set of  $G'_2$ for $j=1,\dots,r-1$ and $v_2$ is filled in the $r$th copy, so we can repeat the process just described for $G'_1$ for $G'_2$, and then additional $G'_i$ as needed to fill all vertices of $G$. This together with Proposition \ref{lem1} proves that $\Z(K_r \Box G)\le (r-1)\Z(G)+1.$ 
\end{proof}

\begin{question}\label{q:starcliquecoverse}
Does converse of Theorem \ref{t:scp1} hold? That is, does 
\[ \Z(G \Box G') = \Z(G)\Z(G')+1, \]
imply that $G$ is a complete graph of order at least two and $G'$ is a star-clique path or vice versa?
\end{question}

Computations in \emph{Sage} \cite{sage} have established that for $n\le 9$, every graph $G’$ of order $n$ satisfying $\Z(K_2\Box G’)=\Z(K_2)\Z(G’)+1$ is a star-clique path.

\section{Concluding remarks}\label{s:conclude}

In this section we summarize 
values of the fort number, fractional zero forcing number, and zero forcing number for various families of graphs in Table \ref{t:ftzstarzvalues}. We also discuss  a possible relationship between maximum nullity $\M$ and $\ft,\Z^*$  in Section \ref{s:ft-less-N} and other open questions  in Section \ref{subsec:open}.

\subsection{Summary of parameter values for various graph families}

In Table \ref{t:ftzstarzvalues}, the result number listed (from within this paper) establishes the fort number and fractional zero forcing number.  In each case, the zero forcing number can be found in at least one of the numbered result,  \cite[Theorems 9.5, 9.12]{HogLinShad}, or  the listed reference.

\begin{table}[h!]
\renewcommand{\arraystretch}{1.3}
\begin{center}
\noindent \begin{tabular}{| l | c | c| c| c| c|}
\hline
result \#&  $G$ & order & $\ft(G)$ & $\Z^*(G)$ & $\Z(G)$\\
\hline
\ref{ex:path} & $P_n$  & $n$& $1$ & $1 $ & $1$ \\
\hline
\ref{ex:complete} & $K_n$  & $n$& $\lf \frac n 2\rf $ & $ \frac n 2 $ & $n-1$ \\
\hline
\ref{ex:Kpq} & $K_{p,q},\, 2\le p,q$  & $p+q$& $\lf \frac p 2\rf +\lf \frac q 2\rf$ &  $\frac{p+q}{2}$  & $p+q-2$ \\
& $K_{1,q},\, 2\le q$  & $1+q$& $\lf \frac q 2\rf$ & $ \frac {q} 2 $  & $q-1$ \\
\hline
\ref{ex:empty} & $\OL{K_n}$  & $n$& $n$ & $n$ & $n$ \\
\hline
\ref{ex:cycle} & $C_{2k}$   & $2k$& $2$ & $2 $ & $2$ \\
 & $C_{2k+1}$   & $2k+1$& $1$ & $ \frac{2k+1}{k+1}$ & $2$ \\
 \hline
\ref{ex:double_foliation} & $G'\circ 2K_1$  & $3|V(G')|$& $|V(G')|$ & $|V(G')|$ & $|V(G')|$ \\
\hline
\ref{ex:triangles}, \ref{p:Zcorona} 
& $G'\circ K_2$  & $3|V(G')|$& $|V(G')|$ & $|V(G')|$ & $|V(G')|+\Z(G')$ \\
\hline
\ref{ex:petersen} & Petersen graph  & $10$ & $2$ & $\frac 5 2$ & $5$ \\
\hline
\ref{ex:hypercube}, \cite{aim} & $Q_d$  & $2^d$&  & $\frac{2^d}d$ & $2^{d-1}$ \\\hline
\ref{ex:sK3veeK1} & $sK_3\vee K_1,\, s\ge 2$  & $3s+1$& $s+1$ & $\frac{3s}2$ & $2s+1$ \\
\hline
\ref{ex:sK2veeK2} & $sK_2\vee K_2,\, s\ge 2$  & $2s+2$& $s+1$ & $s+1$ & $s+2$ \\
\hline
\ref{thm:LL} & $T\in\T$  & & $\frac{\ell(T)} 2$ & $\frac{\ell(T)} 2$ & $\frac{\ell(T)} 2$
\\
\hline
\ref{ex:2parpath-evencycle} & polygonal path, even order  & $2k$ & $2$ & $2$ & $2$ \\
\hline
\ref{ex-sharp-bound}, \ref{boxfort} & $K_{2m} \Box K_{2m}, m \ge 2$ & $4m^2$ & $m^2$ & $m^2$ & $4m^2-4m+2$\\
\hline
\end{tabular}
\caption{Summary of values of fort number and fractional zero forcing number for families of graphs. }\label{t:ftzstarzvalues}
\end{center}
\end{table}

\subsection{A possible relationship between maximum nullity $\M$ and $\ft$ and/or $\Z^*$}\label{s:ft-less-N}
 Recall that the zero forcing number $\Z(G)$ is a well-known upper bound on the maximum nullity of a graph $\M(G)$. It is natural to explore the relationship between $\M(G)$ and the new parameters $\ft(G)$ and $\Z^*(G)$.  We are unaware of any graphs for which $\ft(G) > \M(G)$, but we point out that $\M(G)= \Z(G) \geq \Z^*(G) \geq \ft(G)$ for all graphs order 7 or less \cite {small}  and for all but one of the graph families listed in Table \ref{t:ftzstarzvalues}.  For all but four of  these graph families, the values of $\M(G)$ and $\Z(G)$  appear in \cite[Theorems 9.5, 9.12]{HogLinShad} and/or \cite{aim}.  The exceptions are  $G'\circ 2K_1$, $K_1\vee sK_3$, $K_2\vee sK_2$, and $G'\circ K_2$.  
 
 There is a standard technique that can be used to show $\M(G)=\Z(G)$ for $G'\circ 2K_1$, $K_1\vee sK_3$ and $K_2\vee sK_2$ and to show that $\Z^*(G) <\M(G)$ for $G'\circ K_2$.  
Recall that the \emph{minimum rank} of $G$ is  
$\mr(G)=\min\{\rank A : A\in \mathcal{S}(G)\}$, and $\mr(G)+\M(G)=|V(G)|$.
Viewing a graph  $G$ as the (possibly overlapping) union of 
 $G_i, i=1,\dots,t$,  it is known  \cite{HogLinShad} that
 \[ \mr\left( \bigcup_{i=1}^t G_i \right) \leq \sum_{i=1}^t \mr(G_i). \]

 For $G=G'\circ 2K_1$, we see that $G$ is the union of stars, one for each vertex $v$ of $G'$ centered on $v$ and having as leaves all neighbors of $v$ in $G$.  Since $\mr(K_{1,q})=2$ for $q\ge 2$, we have $\mr(G)\le 2|V({G'})|$, which implies $\M(G)\ge |V({G'})|=\Z(G)$.

 For $G=K_1\vee sK_3$, we see that $G$ is the union of $s$ copies of $K_4$, one for each $K_3$ together with the vertex of the $K_1$.  Since $\mr(K_{4})=1$ and $s\ge 2$, we have $\mr(G)\le s$, which implies $\M(G)\ge 2s+1=\Z(G)$.  The case $G=K_2\vee sK_2$ is similar.

 Finally consider $G=G'\circ K_2$, where $G$ is the union of $G'$ and $|V(G')|$ copies of $K_3$.  Thus $\mr(G)\le \mr(G')+|V(G')|$.  This implies $\M(G)\ge \M(G')+|V(G')|$, which may be less than $|V(G')|+\Z(G')$.  However, $\ft(G)=\Z^*(G)=|V(G')|<\M(G)$.

 \begin{question} \label{quest:M_versus_ft}
    Is $\M(G)\ge \ft(G)$ for every graph $G$?   Is $\M(G)\ge \Z^*(G)$ for every graph $G$?
\end{question}
  While we cannot prove even the weaker of the two inequalities,   $\M(G) \ge \ft(G)$, we can prove that the fort number is at most the maximum nullity among combinatorially symmetric matrices described by the graph.

A matrix $A$ is \emph{combinatorially symmetric} if $a_{ij}\ne 0$ if and only if $a_{ji}\ne 0$. The \emph{graph} $\G(A)$ of  a combinatorially symmetric matrix  $A$  is the graph with vertices $\{1,\dots,n \}$  and edges $\{ ij :a_{ij} \ne 0,  1 \le i <j \le n \}$; 
whenever we write $\G(A)$, it is assumed that $A$ is combinatorially symmetric. 
For a graph $G$ with $V(G)=\{1,2,\ldots,n\}$, 
 the maximum nullity of combinatorially symmetric matrices described by $G$, $\n(G)=\max\{\nul A :A\in\Rnn, \G(A)=G\}$,  has not been as widely studied as $\M(G)$, but it  is known that $\n(G)\le\Z(G)$ \cite{aim}. 

The next result connects forts and null vectors.  For a real vector $\bx=[x_i]$, the \emph{support} of $\bx$ is the set of indices $i$ for
which $x_i\ne 0$.
 \begin{theorem}\label{sparkthm}
   {\rm\cite{sparkx}}  Let $G$ be a graph of order $n$ and let $\bzero\ne \bx\in{\mathbb{R}^n}$.
  There exists a matrix $A\in S(G)$ such that $A\bx=\bzero$ if
and only if the support of $\bx$ is a fort of $G$. 
\end{theorem}
For a graph $G$ and fort $F$ of $G$, the \emph{incidence vector} $\bv=[v_j]$ of $F$ is the vector such that
$v_j = 1$ if $j\in F$ and 0 if $j\not\in F$.

\begin{proposition}\label{thm:N_ft}
Let $G$ be a graph with disjoint forts $F_{1},\ldots,F_{k}$, where $1\leq k\leq n$  and let $\mathbf{v}_{1},\ldots,\mathbf{v}_{k}$ be the corresponding incidence vectors.  Then  $\mathbf{v}_{1},\ldots,\mathbf{v}_{k}$ are null vectors for some $A\in\Rnn$ such that $\G(A)=G$. Thus, $\n(G)\geq\ft(G)$.
\end{proposition}
\bpf Without loss of generality, assume the vertices within each fort are numbered consecutively, so that $F_i=\{v_{\ell_{i-1}+1},\dots, v_{\ell_i}\}$ where $\ell_0=0$.
For $i=1,\dots,k$, apply Theorem \ref{sparkthm} to choose matrices $A_i\in\Rnn$ such that $\G(A_i)=G$ and $A_i\bv_i=\bzero$. 
   Construct matrix $A$ by choosing columns $\ell_{i-1}+1,\dots,\ell_i$ from $A_i$ (and columns $\ell_{k}+1,\dots,\ell_n$ from $A_k$ if necessary).  Then $\G(A)=G$ and $A\bv_i=\bzero$ for $i=1,\dots,k$ because $A\bv_i=A_i\bv_i$ is the sum of the columns   associated with the indices in $F_i$ and the forts are disjoint. 
  The last statement is immediate. 
\epf

{\subsection{Summary of open questions} \label{subsec:open}

This work provides a basis for many ripe and interesting questions. The conjectured lower bound on zero forcing number of a Cartesian product (Conjecture \ref{con:main}) remains open in general, and Question \ref{q:starcliquecoverse} asks whether Theorem \ref{t:scp1} characterizes all graphs attaining equality in the bound in Conjecture \ref{con:main}.

Beyond the main conjecture, there are plenty of directions of study for these new parameters.  Question \ref{fort_question}
asks whether $\Z^*(G) = \Z(G)$ if and only if $\ft(G)=\Z(G)$ for all graphs $G$.
Question \ref{quest:ccr_zstar} asks what color-change rule might be used to compute $\Z^*(G)$. Such a rule could potentially make it easier to compute or bound $\Z^*(G)$ without enumerating many forts.
Finally, Question \ref{quest:M_versus_ft} asks whether $\ft(G)$ and/or $\Z^*(G)$ provide a lower bound for $\M(G)$; this is interesting as the original motivation for the classical zero forcing number $\Z(G)$ was to provide an upper bound on $\M(G)$.
}

\begin{table}[h!]
    \centering
    \begin{tabular}{ll}
        \toprule
        \textbf{Reference} & \textbf{Short description} \\
        \midrule
        Conjecture \ref{con:main} & A Vizing-like lower bound on zero forcing number of a Cartesian product. \\ 
        Question \ref{quest:ccr_zstar} & Color-change rule for computing $\Z^*(G)$. \\
        Question \ref{fort_question} & Relationship between $\Z^*(G) = \Z(G)$ and $\ft(G)=\Z(G)$. \\
        Question \ref{q:starcliquecoverse} & Characterization of graphs attaining equality in the bound in Conjecture \ref{con:main}. \\
        Question \ref{quest:M_versus_ft} & Whether $\ft(G)$ and/or $\Z^*(G)$ bound $M(G)$ from below. \\
        \bottomrule
    \end{tabular}
    \caption{A summary of open questions presented throughout this article.}
\label{tab:openquestions}
\end{table}


\newpage \section*{Acknowledgments}
The authors and this research were supported by the American Institute of Mathematics (AIM) and grants from the National Science Foundation (NSF).





\end{document}